\documentclass[11pt]{amsart}
\usepackage{}
\usepackage{amsfonts}
\usepackage{amsfonts,latexsym,rawfonts,amsmath,amssymb,amsthm}
\usepackage[plainpages=false]{hyperref}
\usepackage{mathrsfs}
\RequirePackage{amsmath} \RequirePackage{amssymb}
\usepackage{fancyhdr}
\usepackage{graphicx}
\usepackage{amscd,latexsym,amsthm,amsfonts,amssymb,amsmath,amsxtra}
\usepackage{amscd,mathrsfs,latexsym,amsmath,amsthm,amssymb,amsxtra,}
\usepackage{latexsym,amsmath,amsthm,amssymb,amsxtra,mathrsfs}
\usepackage[all]{xy}



\numberwithin{equation}{section}

\newcommand{\beq}{\begin{equation}}
\newcommand{\eeq}{\end{equation}}
\newcommand{\beqs}{\begin{eqnarray*}}
\newcommand{\eeqs}{\end{eqnarray*}}
\newcommand{\beqn}{\begin{eqnarray}}
\newcommand{\eeqn}{\end{eqnarray}}
\newcommand{\beqa}{\begin{array}}
\newcommand{\eeqa}{\end{array}}

\def\lra{\longrightarrow}

\def\bc{\begin{center}}
\def\ec{\end{center}}

\def\begeq{\begin{equation}}
\def\endeq{\end{equation}}
\def\and{\quad{\rm and}\quad}

\let\lra=\longrightarrow

\def\mapright\#1{\,\smash{\mathop{\lra}\limits^{\#1}}\,}

\newtheorem{prop}{Proposition}[section]
\newtheorem{theo}[prop]{Theorem}
\newtheorem{lem}[prop]{Lemma}
\newtheorem{claim}[prop]{Claim}
\newtheorem{cor}[prop]{Corollary}
\newtheorem{rem}[prop]{Remark}

\newtheorem{defi}[prop]{Definition}
\newtheorem{conj}[prop]{Conjecture}
\newtheorem{q}[prop]{Question}

\begin{document}

\title{Bergman Kernels for a  sequence of almost K\"{a}hler-Ricci solitons}

\author {Wenshuai Jiang}
\author  {Feng Wang }
\author { Xiaohua $\text{Zhu}^*$}

\thanks {* Partially supported by the NSFC Grants 11271022 and 11331001}

\address{ Wenshuai Jiang \\School of Mathematical Sciences, Peking University,
Beijing, 100871, China}

\address{ Feng Wang\\School of Mathematical Sciences, Peking University,
Beijing, 100871, China}

\address{ Xiaohua Zhu\\School of Mathematical Sciences and BICMR, Peking University,
Beijing, 100871, China\\
 xhzhu@math.pku.edu.cn}

\bibliographystyle{plain}

\begin{abstract} In this paper,   we give  a lower bound of  Bergman kernels for  a sequence of  almost K\"{a}hler-Einstein  Fano manifolds,  or more general, a  sequence of  Fano manifolds with
 almost K\"{a}hler-Ricci solitons.  This  generalizes  a result by Donaldson-Sun,  Tian for  K\"{a}hler-Einstein  manifolds sequence  with positive scalar curvature. As an application of our result,
we prove that the Gromov-Hausdorff limit  of  sequence is   homomorphic to a log terminal $Q$-Fano variety which admits a  K\"{a}hler-Ricci soliton  on its smooth part.
\end{abstract}

\maketitle

\tableofcontents

\date{}
\maketitle
\bibliographystyle{plain}

\section{Introduction}

 Let $(M^n, g)$ be an $n$-dimensional  Fano manifold  with its K\"ahler form $\omega_g$ in $2\pi c_1(M)$. Then  $g$  induces a   Hermitian metric  $h$ of
 the  anti-canonical  line bundle $K_M^{-1}$
  such that ${\rm Ric }(K_M^{-1}, h)=\omega_g$.   Also $h$ induces a  Hermitian metric  ( for simplicity, we  still use the notation $h$ ) of  $l$-multiple
line   bundle   $K_M^{-l}$.  As usual,   the   $L^2$-inner product on $H^0(M,K_M^{-l})$  is given by
\begin{align}\label{inner-product}  ( s_1,s_2)=\int_M\langle s_1,s_2\rangle_h d{\rm v}_g ,~ \forall~  s_1, s_2~\in H^0(M,K_M^{-l}).
\end{align}
 Choosing  a unit orthogonal  basis  $\{s_i\}$ of  $H^0(M,K_M^{-l})$ with respect to the inner product $(\cdot,\cdot)$ in (\ref{inner-product}),   we define
the Bergman kernel of   $(M,K_M^{-l}, h)$  by
\begin{align}
\rho_l(x)=\Sigma_i |s_i|^2_h(x).\notag
\end{align}
Clearly,  $\rho_l(x)$ is independent of  the choice  of basis $\{s_i\}$.
 In \cite{T4}, Tian proposed a conjecture for the existence of uniform  lower bound of $\rho_l(x)$:

 \begin{conj}\label{tian-conjecture} Let  $\{(M_i,   g^i)\}$  be a sequence of   $n$-dimensional K\"{a}hler-Einstein
  manifolds with constant  scalar curvature $n$.  Then   there exists
   an integer number $l_0$ such that for any integer $l>0$ there exists a uniform constant $c_l >0$ with property:
\begin{align}\label{lower-bound-bergman}
\rho_{ll_0}(M_{i}, g^{i})\geq c_l.\end{align}
Here $c_{l}$ depends only on $l, n$.
\end{conj}

The above  conjecture was recently proved by Donaldson-Sun \cite{DS} and Tian \cite{T6},  independently. \footnote{Phong-Song-Strum extended the result to a  sequence of   K\"{a}hler-Ricci solitons  lately \cite{PSS}.} The main  idea  in their  proofs  is to use the H\"ormander $L^2$-estimate to  construct   peak  holomorphic sections  by solving  $\overline\partial$-equation.   This   idea can  go back to Tian's  orginal work \cite{T1}  (see also a survey paper by him \cite{T4}).   In fact he used the idea to prove   the conjecture  for   the  K\"{a}hler-Einstein  surfaces   more than twenty years ago \cite{T2}.

The estimate  (\ref{lower-bound-bergman})  is usually called   the partial $C^0$-estimate.  Very recently,   (\ref{lower-bound-bergman})  was generalized to   a  sequence of  conical  K\"{a}hler-Einstein  manifolds by Tian \cite{T5}.  As an application of (\ref{lower-bound-bergman}) he gives  a  proof of the  famous  Yau-Tian-Donaldson's conjecture for the existence problem of K\"ahler-Einstein metrics with positive scalar curvature. Chen-Donaldson-Sun also gives  a proof of  the conjecture    independently  \cite{CDS}.

 \begin{theo}[Tian, Chen-Donaldson-Sun]\label{TCDS} A Fano manifold admits a  K\"{a}hler-Einstein metric if and only if it is $K$-stable.
\end{theo}

The $K$-stability was first introduced by Tian \cite{T3} and it was  reformulated by Donaldson  in terms of   test-configurations \cite{Do}.

In  this  paper,  we want to generalize the estimate  (\ref{lower-bound-bergman}) to  a sequence of  almost K\"{a}hler-Einstein  Fano manifolds \cite{TW}, or more general, a  sequence of  almost K\"{a}hler-Ricci solitons (see Definition  \ref{almost-kr-solitons} in Section 7).    Namely, we prove

\begin{theo}\label{main-theorem-wang-zhu}
Let $\{(M_i,   g^i)\}$    be a sequence of     almost K\"{a}hler-Einstein  Fano  manifolds   ( or a  sequence of  Fano manifolds with  almost K\"{a}hler-Ricci solitons)  with dimension  $n\ge 2$.  Then   there exists   an integer number $l_0$ such that for any integer $l>0$ there exists a uniform constant $c_l>0$ with property:
\begin{align}\label{lower-bound-bergman-tian}
\rho_{ll_0}(M_{i}, g^{i})\geq c_l.
\end{align}
Here  the constant $c_l$ depends only on $ l, n$, and some  uniform geometric  constants (cf. Section 9).
\end{theo}

As in the proof of Theorem \ref{TCDS},   we  need to construct   peak  holomorphic sections  by solving  $\overline\partial$-equation to prove
Theorem \ref{main-theorem-wang-zhu}.  Because  there is a lack of local  strong convergence of  $\{(M_i,  g^i)\}$,   we  shall smooth
the sequence  to  approximate  the original one by  Ricci flow as in \cite{TW}, \cite{WZ2}.  This approximation   will depend  on  points in the
Gromov-Hausdorff  limit space of   $\{(M_i,  g^i)\}$,   so it depends  on  the time $t$ in the   Ricci flow.  Thus  we  need
to give estimates for the scalar curvatures and K\"ahler  potentials along the flow for small time $t$ (cf. Section 2, 3, 7).   Another technical part is in the  construction of    peak  holomorphic sections by using  the rescaling method as in \cite{DS},  \cite{T6},  which will depend on the choice of K\"ahler metrics evolved in the  Ricci flow  in our case (cf. Section 5, 6, 7).

 Together with  the main results in \cite{WZ1} and \cite{WZ2},  Theorem \ref{main-theorem-wang-zhu} implies

\begin{cor}\label{algebaric-structure}
Let $\{(M_i,   g^i)\}$    be a sequence of  almost K\"{a}hler-Einstein  Fano  manifolds  ( or a  sequence of  Fano manifolds  with almost K\"{a}hler-Ricci solitons) with dimension $n\ge 2$.  Then  $\{(M_i,   g^i)\}$ (maybe replaced by   a  subsequence of   $\{(M_i,   g^i)\}$)  converges
 to  a  metric space  $(M_\infty,g_\infty)$ in Gromov-Hausdorff topology with   properties:

 i)~ The codemenison of singularities  of   $(M_\infty,g_\infty)$  is  at least $4$;

ii)~     $g_\infty$  is  a K\"{a}hler-Einstein metric ( or a K\"{a}hler-Ricci soliton)   on the regular part of $ M_\infty$;

iii)~    $M_\infty$  is homomorphic to a log terminal $Q$-Fano variety.

\end{cor}

 In case of  K\"{a}hler-Einstein  manifolds with positive scalar curvature,  we note that  i) and ii) in  Corollary \ref{algebaric-structure} follow from  the Cheeger-Colding-Tian compactness theorem  \cite{CCT}.   Donaldson-Sun proved the part iii )  except  log    terminal property   \cite{DS}  (also see  \cite{L1}).
Since any  $Q$-Fano variety, which  admits a  K\"ahler-Einstein metric,  is  automatically    log terminal according to Proposition 3.8 in \cite{BBEGZ}\footnote{ The result  also holds for a  $Q$-Fano variety, which  admits a  K\"ahler-Ricci soliton,  from the proof of Proposition 3.8.}, Thus the part iii) is true. A normal variety $M$ is called  $Q$-Fano if the restriction of $\mathcal O_{\mathbb CP^N }(1)$ is a multiple of $K_{M}^{-1}$ on the smooth part of $M$.  The log terminal means that  there exists a resolution
$\pi:~\tilde M\to M$ such that
$K_{\tilde M}=\pi^*K_M+\sum a_iD_i, $ where $a_i>-1$, $\forall ~i.$

There are important  examples   of   almost K\"{a}hler-Einstein  metrics and  almost K\"{a}hler-Ricci solitons:

1)~ Tian and  B. Wang  constructed a  family  of  almost K\"ahler-Einstein  metrics   $g_t$ $(t\to 1)$ arising from  solutions of certain complex Monge-Amp\`ere equations on a Fano manifold with the Mabuchi's $K$-energy bounded  below \cite{TW}.

2) ~Tian  constructed  a family of   almost K\"ahler-Einstein  metrics  $g_t$ $(t\to 1)$ modified  from
 conical K\"ahler-Einstein  metrics  on a Fano manifold  whose corresponding conical angles go to $2\pi$ \cite{T5}.

3)~  F. Wang and Zhu   constructed  a  family  of  almost K\"{a}hler-Ricci solitons  $g_t$ $(t\to 1)$  arising from  solutions of certain complex Monge-Amp\`ere equations
 on a Fano manifold with the modified   $K$-energy bounded  below  \cite{WZ1},  \cite{WZ2}.

  By Theorem \ref{main-theorem-wang-zhu} and Corollary \ref{algebaric-structure},    we have

\begin{cor}\label{k-energy-bounded}
Let $g_t$ $(t\to 1)$ be  a  family  of  almost K\"ahler-Einstein  metrics ( or almost
 K\"ahler-Ricci  solitons ) on a  Fano manifold  $M$ constructed  above 1), 2), 3).  Then  there exists   an integer number $l_0$ such that for any integer $l>0$ there exists a uniform constant $c_l>0$  independent of  $t$ with property:
\begin{align}\label{lower-k-energy}
\rho_{ll_0}(M,g_t)\geq c_l>0.
\end{align}
 Moreover,  there exists a sequence  $\{(M,   g_t)\}$ which   converges
 to  a  metric space  $(M_\infty,g_\infty)$ in Gromov-Hausdorff topology with   properties i),ii),iii) in Corollary \ref{algebaric-structure}.

\end{cor}

 It was proved recently  by Li that    the lower boundedness  of $K$-energy  is  equivalent  to  the
 $K$-semistablity  [L2].  Li's proof depends on the construction of test-configurations  in  Theorem \ref{TCDS}  by studying  conical  K\"{a}hler-Einstein metrics.
 It is reasonable to believe that there is   an analogy of Li's result to describe  the  modified   $K$-energy  in sense of  modified  $K$-semistability.

 \begin{q}\label{question}  Is  there a direct proof (without using   conical  K\"{a}hler-Einstein metrics as in the proof of  Theorem \ref{TCDS})  for that  the $K$-stability implies the   $K$-energy bounded  below?
 \end{q}

A same question was  proposed    by   Paul  in his recent paper \cite{Pa}.  He proved there that the $K$-stability is equivalent to the properness of  $K$-energy in  the space of K\"ahler metrics induced
by  the Bergman Kernels.  Thus  as pointed by Tian in \cite{T6}, \cite{T7} (also see \cite{T3}),   (\ref{lower-k-energy}) will give  a  new proof  of Theorem \ref{TCDS}   if  the answer to Question \ref{question} is positive.

Finally, let us to  state our organization to the paper.  In Section 2,    we give  some  estimates  for   scalar curvatures and K\"ahler  potentials along  the Ricci  flow,     then, in Section 3,  we use  them to give  the $C^0$-estimate and  the gradient estimate for holomorphic sections on multiple line bundles of $K_M^{-1}$.
Section 4 is devoted  to construct almost    peak  holomorphic sections by  using the trivial bundle on the tangent cone.
The peak  holomorphic sections,  which  depend on time $t$,  will be constructed   in Section 5. Theorem \ref{main-theorem-wang-zhu} will be proved in Section 6, 7,  according to  almost K\"{a}hler-Einstein   metrics and  almost K\"{a}hler-Ricci solitons, respectively,   while  its proof is completed  in Section 9.  In Section 8,  we prove Corollary  \ref{algebaric-structure}.

\vskip3mm

\noindent {\bf Acknowledgements.}   The authors would like to thank
professor G.  Tian for his encourage and  many valuable discussions   during working on the paper.

\vskip3mm

\section{Estimates  from  K\"{a}hler Ricci flow}

In this section,  we give some necessary estimates  for  the scalar curvatures and K\"ahler  potentials along  the K\"ahler-Ricci  flow.
 Let $(M,g)$ be an $n$-dimensional
Fano manifold  with its K\"ahler form $\omega_g$ in $2\pi c_1(M)$.    Let $g_t=g(\cdot,t)$ be a  solution of normalized K\"{a}hler Ricci flow,
\begin{align}\left\{\begin{aligned}\label{Ricci flow}
&\frac{\partial}{\partial t}g=-{\rm{ Ric }}(g)+g,\\
&g_0=g(\cdot,0)=g.
 \end{aligned} \right.
\end{align}
Recall an  estimate for  Sobolev constants of $g_t$   by Zhang \cite{Zh}.

\begin{lem}\label{Sobolev}
Let $g_t$ be the solution of (\ref{Ricci flow}). Suppose that   there exists a Sobolev  constant $C_s$  of $g$  such that  the  following  inequality holds,
\begin{align}\label{sobolev-constant-g}
(\int_M f^{\frac{2n}{n-1}}{\rm{dv}}_g)^\frac{n-1}{n}\leq C_s(\int_M f^2{\rm{d v}}_g+\int_M|\nabla f|^2{\rm{dv}}_g), ~\forall ~f\in C^1(M).
\end{align}
Then there exist two uniform constants  $A=A(C_s, -\inf_M R(g),V)$ and $ C_0=C_0(C_s, -\inf_M R(g),V)$ such that  for any $f\in C^1(M)$ it holds
\begin{align}
(\int_M f^{\frac{2n}{n-1}}{\rm{dv}}_{g_t})^\frac{n-1}{n}\leq A(\int_M(|\nabla f|^2+(R_t+C_0)f^2){\rm{dv}}_{g_t},
\end{align}
where $R_t$  are  scalar curvatures of $g_t$.

 \end{lem}

By using the Moser  iteration method, we have

\begin{lem}\label{iteration} Let $\Delta=\Delta_t$  be  Lapalace operators associated to  $g_t$.  Suppose that
     $f\ge 0$   satisfies
\begin{align}\label{equation}
(\frac{\partial}{\partial t}-\Delta)f\leq af, ~\forall~t\in (0,1),
\end{align}
where  $a\geq 0$ is a constant.   Then for any  $t\in (0,1)$, it holds
\begin{align}\label{sobolev-iteration-1}
&\sup_{x\in M}f(x,t)\notag\\
&\leq \frac{C}{t^\frac{n+1}{p}}(\int_{\frac{t}{2}}^t\int_M|f(x,\tau)|^p \text{{\rm{dv}}}_{g_{\tau}}d\tau)^\frac{1}{p},
\end{align}
where $C= C(a,p,C_s, -\inf R(g),V)$ , $p\ge 1$ and  $C_s$ is the Sobolev constant of $g$ in (\ref{sobolev-constant-g}).
\end{lem}

\begin{proof}
Multiplying both sides of (\ref{equation}) by $f^p$,  we have
\begin{align}
\int_Mf^pf'_\tau  \text{{\rm{dv}}}_{g_{\tau}} -\int_Mf^p\Delta f   \text{{\rm{dv}}}_{g_{\tau}}\leq a \int_Mf^{p+1}.\notag
\end{align}
Taking integration by parts, we  get
\begin{align}
&\frac{1}{p+1}\int_M (f^{p+1})'_\tau    \text{{\rm{dv}}}_{g_{\tau}} +\frac{4p}{(p+1)^2}\int_M |\nabla f^{\frac{p+1}{2}}|^2   \text{{\rm{dv}}}_{g_{\tau}}\notag\\
&\leq a   \int_Mf^{p+1}    \text{{\rm{dv}}}_{g_{\tau}}.\notag
\end{align}
Since
\begin{align}
\frac{d}{d\tau}\int_M f^{p+1}   \text{{\rm{dv}}}_{g_{\tau}}=\int_M (f^{p+1})'_\tau   \text{{\rm{dv}}}_{g_{\tau}}+\int_M f^{p+1}(n-R)   \text{{\rm{dv}}}_{g_{\tau}},\notag
\end{align}
we deduce
\begin{align}
&\frac{1}{p+1}\frac{d}{d\tau}\int_M f^{p+1}     \text{{\rm{dv}}}_{g_{\tau}}+\frac{1}{p+1}\int_M Rf^{p+1}   \text{{\rm{dv}}}_{g_{\tau}}\notag\\
&+\frac{4p}{(p+1)^2}\int_M |\nabla f^{\frac{p+1}{2}}|^2
\leq (a+\frac{n}{p+1})\int_Mf^{p+1}    \text{{\rm{dv}}}_{g_{\tau}}.\notag
\end{align}
It turns
 \begin{align}\label{gradient-integral}
&\frac{d}{d\tau}\int_M f^{p+1}    \text{{\rm{dv}}}_{g_{\tau}}+\int_M (R+C_0)f^{p+1}    \text{{\rm{dv}}}_{g_{\tau}}+2\int_M |\nabla f^{\frac{p+1}{2}}|^2\notag\\
&\leq ((p+1)a+n+C_0)\int_Mf^{p+1}   \text{{\rm{dv}}}_{g_{\tau}}.
\end{align}

For any $0\leq \sigma'\leq \sigma\leq 1$,  we define
\begin{align}
\psi(\tau)=\begin{cases}
 0, \tau \leq \sigma'  t\\
 \frac{\tau-\sigma' t}{(\sigma-\sigma')t}, \sigma' t\leq \tau\leq \sigma t\\
 1, \sigma t\leq \tau\leq t.
\end{cases}\notag
\end{align}
Then by (\ref{gradient-integral}),  we have
\begin{align}
&\frac{d}{d\tau}(\psi \int_M f^{p+1}    \text{{\rm{dv}}}_{g_{\tau}})+\psi \int_M[(R+C_0)f^{p+1}+2|\nabla f^{\frac{p+1}{2}}|^2]   \text{{\rm{dv}}}_{g_{\tau}}\notag\\
&\leq [\psi((p+1)a+n+C_0)+\psi']\int_Mf^{p+1}   \text{{\rm{dv}}}_{g_{\tau}}.\notag
\end{align}
It follows
\begin{align}
&\sup_{\sigma t\leq \tau\leq t}\int_M f^{p+1}   d\text{{\rm{v}}}_{g_{\tau}}+\int_{\sigma t}^t\int_M[(R+C_0)f^{p+1}+2|\nabla f^{\frac{p+1}{2}}|^2] d\text{{\rm{v}}}_{g_{\tau}}\notag\\
&\leq ((p+1)a+n+C_0+\frac{1}{(\sigma-\sigma')t})\int_{\sigma' t }^t\int_Mf^{p+1}  d\text{{\rm{v}}}_{g_{\tau}}.\notag
\end{align}
Thus by  Lemma \ref{Sobolev},  we get
\begin{align}\label{iteration-formula}
&\int_{\sigma t}^t \int_M f^{(p+1)(1+\frac{1}{n})}    \text{{\rm{dv}}}_{g_{\tau}}\notag\\
&\leq (\int_{\sigma t}^t\int_M  f^{p+1}    \text{{\rm{dv}}}_{g_{\tau}})^{\frac{1}{n}}(\int_M  f^{(p+1)\frac{n}{n-1}})^{\frac{n-1}{n}}\notag\\
&\leq [(\sup_{\sigma t\leq \tau\leq t}\int_M f^{p+1}    \text{{\rm{dv}}}_{g_{\tau}}]^{\frac{1}{n}} \int_{\sigma t}^t A\int_M[(R+C_0)f^{p+1}+2|\nabla f^{\frac{p+1}{2}}|^2]    \text{{\rm{dv}}}_{g_{\tau}} \notag\\
&\leq A((p+1)a+n+C_0+\frac{1}{(\sigma-\sigma')t})^\frac{n+1}{n}(\int_{\sigma' t}^t\int_Mf^{p+1}   \text{{\rm{dv}}}_{g_{\tau}})^{\frac{n+1}{n}}.
\end{align}
By choosing  $\sigma'=\frac{1}{2}+\frac{1}{4}\sigma_k, \sigma=\frac{1}{2}+\frac{1}{4}\sigma_{k+1}$, where $\sigma_k=\sum_{l=0}^k (\frac{1}{2}))^l-1$,  and  replacing $p$ by $p_{k+1}=(p_k+1)^\frac{n+1}{n}-1$ with $p_0=p\ge 0$ in (\ref{iteration-formula}), then  iterating  $k$  we   will get the  desired
 estimate (\ref{sobolev-iteration-1}).
\end{proof}

By Lemma \ref{iteration}, we prove

\begin{prop}\label{estimate-u-R}  Let  $u=u_t$   and $R=R_t$ be   Ricci potentials  and scalar curvatures  of  solutions  $g_t$   in  (\ref{Ricci flow}), respectively.
Suppose that $(M,g)$  satisfies
\begin{align}\label{Ric-con}
{\rm Ric}(g) \geq -\Lambda^2 g~{\rm and}~{\rm diam}(M,g)\leq D.
\end{align}
  Then there  exists  a  constant  $C(n,,\Lambda,D)$ such that
\begin{align}\label{gradient-estimate-2}
&|\nabla u|^2(x,t)\notag\\
&\leq \frac{C}{t^{(n+1)(n+\frac{3}{2})}}\int_{\frac{t}{2}}^t\int_M |R-n| \text{{\rm{dv}}}_{g_{\tau}} ,~~\forall~ 0< t\leq 1
\end{align}
and
\begin{align}
&|R-n|(x,t)\label{scalar-curvature}\notag\\
&\le \frac{C}{t^{(n+1)(n+\frac{3}{2})+n}}\int_{\frac{t}{2}}^t\int_M|R-n|  \text{{\rm{dv}}}_{g_{\tau}},~~ \forall ~0<t\le 1.
\end{align}
\end{prop}

\begin{proof}
By a direct computation,  we have the  the following evolution formulas  for $|\nabla u|$ and $R$, respectively,
\begin{align}\label{gardient-flow}(\frac{\partial}{\partial t}-\Delta)|\nabla u|^2
&=\Delta |\nabla u|^2-|\nabla\nabla u|^2-|\nabla\bar{\nabla}u|^2+|\nabla u|^2\le |\nabla u|^2
 \end{align}
  and
 \begin{align}\label{r-flow}
(\frac{\partial}{\partial t}-\Delta)R&=\Delta R+R-n+|\text{Ric}(g)-g|^2.
\end{align}
It follows
\begin{align}\label{r-gradient-flow}
&(\frac{\partial}{\partial t}-\Delta)(R+n\Lambda+|\nabla u|^2)\notag\\
&=R-n-|\nabla\nabla u|^2+|\nabla u|^2\leq R+n\Lambda+|\nabla u|^2.
\end{align}
Note that $R(g_t)+n\Lambda\ge 0$  by the maximum principle.
 It was proved in   \cite{Ji} that there exists a uniform constant  $C=C(\Lambda,D)$ such that
 $$\int_0^1\int_M(R+n\Lambda+|\nabla u|^2)d\text{{\rm{v}}}_g dt\le C.$$
 Then  by  Lemma \ref{iteration},  we obtain
\begin{align}
(R+n\Lambda+|\nabla u|^2)(x,t)\le \frac{C}{t^{n+1}}.
\end{align}
In particular,
\begin{align}\label{gradient-u}
|\nabla u|^2(x,t)\le \frac{C}{t^{n+1}},~~\text{   and   }~~~R\le \frac{C}{t^{n+1}}.
\end{align}

Next we estimate  the $C^0$-norm of $u_t$.  By Lemma \ref{Sobolev}
we have  the Sobolev inequality,
\begin{align}
\nonumber(\int_M f^{\frac{2n}{n-1}}\text{{\rm{dv}}}_{g_t})^\frac{n-1}{n}&\leq A(\int_M(|\nabla f|^2+(R(x,t)+C_0)f^2)\text{{\rm{dv}}}_{g_t})\\
\nonumber &\le A(\int_M(|\nabla f|^2+\frac{C}{t^{n+1}}f^2)\text{{\rm{dv}}}_{g_t}).
\end{align}
The inequality  implies  (cf. \cite{He},  \cite{Ye}),
\begin{align}
{\rm{vol}}(B(x,1))\ge C t^{n(n+1)}, ~~\forall~ x\in M.\notag
\end{align}
Since ${\rm{vol}}(M)=V$,  it is easy to obtain
\begin{align}
{\rm{diam}}(M,g_t)\le \frac{V}{Ct^{n(n+1)}}.\notag
\end{align}
Thus by  (\ref{gradient-u}),   we get
\begin{align}\label{c0-u}
\text{osc}_Mu(x,t)\le \frac{C}{t^{(n+1)(n+\frac{1}{2})}}.
\end{align}

By  (\ref{c0-u}), we can improve  (\ref{gradient-u}) to (\ref{gradient-estimate-2}). In fact,  by applying  Lemma \ref{iteration} to (\ref{gardient-flow}), we have
\begin{align}
|\nabla u|^2(x,t)&\le \frac{C}{t^{n+1}}\int_{\frac{t}{2}}^T\int_M |\nabla u|^2 \text{{\rm{dv}}}_{g_\tau}d\tau\notag\\
\nonumber &=\frac{C}{t^{n+1}}\int_{\frac{t}{2}}^t\int_M -u\Delta u\text{{\rm{dv}}}_{g_\tau}d\tau\notag\\
&\le \frac{C}{t^{n+1}}\text{osc}_{(x,\tau)\in M\times [\frac{t}{2},t]}|u|(x,\tau)\int_{\frac{t}{2}}^t\int_M |R-n|\text{{\rm{dv}}}_{g_\tau}d\tau\notag\\
&\le \frac{C'}{t^{(n+1)(n+\frac{3}{2})}}\int_{\frac{t}{2}}^t\int_M |R-n|\text{{\rm{dv}}}_{g_\tau}d\tau,
\end{align}
where the constant $C'$ depends only on $n,~\Lambda,~D$.
This proves (\ref{gradient-estimate-2}).

To get (\ref{scalar-curvature}),  we  use the evolution equation as  same as (\ref{r-gradient-flow}),
\begin{align}
(\frac{\partial}{\partial t}-\Delta)(|\nabla u|^2+R-n)&=R-n-|\nabla\nabla u|^2+|\nabla u|^2\notag
\\
&\leq |\nabla u|^2+R-n.\notag
\end{align}
Then applying  Lemma \ref{iteration}, we see
\begin{align}
(|\nabla u|^2+R-n)_+&\le \frac{C}{t^{n+1}}\int_{\frac{t}{2}}^t\int_M||\nabla u|^2+R-n|\text{{\rm{dv}}}_{g_\tau}d\tau\notag\\
&\le  \frac{C}{t^{(n+1)(n+\frac{3}{2})}}\int_{\frac{t}{2}}^t\int_M|R-n|\text{{\rm{dv}}}_{g_\tau}d\tau.\notag
\end{align}
Thus  by (\ref{gradient-estimate-2}), it follows
\begin{align}\label{A(t)}
&(R-n)_+\notag\\
&\le \frac{C}{t^{(n+1)(n+\frac{3}{2})}}\int_{\frac{t}{2}}^t\int_M|R-n|d\text{v}_{g_\tau}d\tau:=A(t).
\end{align}
 On the other hand, by the evolution equation (\ref{r-flow}) of $R$,
\begin{align}
(\frac{\partial}{\partial t}-\Delta)R=R-n+|\nabla\bar\nabla u|^2,\notag
\end{align}
we have
\begin{align}
(\frac{\partial}{\partial t}-\Delta)(A(T)+n-R)\le A(T)+n-R.\notag
\end{align}
Hence applying  Lemma  \ref{iteration} again, we get
\begin{align}
&(A(t)+n-R)(x,t)\notag\\
&\le \frac{C''}{t^{n+1}}\int_{\frac{t}{2}}^t\int_M(A(t)+n-R)\text{{\rm{dv}}}_{g_t}dt\notag\\
&\le \frac{C''}{t^{n+1}}\int_{\frac{t}{2}}^T\int_M|n-R|\text{{\rm{dv}}}_{g_\tau}d\tau+\frac{A(t)VC}{t^n}.\notag
\end{align}
 Therefore,  inserting  (\ref{A(t)})  into the above estimate,  we  obtain  (\ref{scalar-curvature}).
\end{proof}

 \section{Estimates for holomorphic sections}

 In this section, we use the estimates  in  Section 2 to give  the $C^0$-estimate and  the gradient estimate for holomorphic sections  with respect to
 $g_t$.    Let $(M^n, g)$ be a Fano manifold  and $L=K_M^{-1}$ its anti-canonical line bundle with induced  Hermitian metric  $h$ by $g$.  We begin with
 the following lemma.

 \begin{lem}\label{section-gradient-estimate}
  Suppose that the Ricci potential $u$ of $g$  satisfies
 \begin{align}\label{Ricci-potential-gradient}
 \|\nabla u\|_{g}\leq 1.
 \end{align}
Then for $s\in H^0(M,L^l)$ we have
\begin{align}\label{section-estiamte-1}
\|s\|_h+l^{-\frac{1}{2}}\|\nabla s\|_{h}\leq C(C_s,n)l^{\frac{n}{2}}(\int_M|s|^2 {\rm dv}_g)^{\frac{1}{2}},
\end{align}
where $C_s$ is the Sobolev constant of $(M,g)$.
\end{lem}

\begin{proof}
Note that
\begin{align}
\Delta |s|_h^2=|\nabla s|_h^2-nl|s|_h^2.\notag
\end{align}
It follows
\begin{align}\label{l^2-lapalace}
  -\Delta |s|_h^2\leq nl|s|_h^2.
\end{align}
Thus applying  the standard Moser iteration  method to (\ref{l^2-lapalace}),  we get
\begin{align}\label{l^2-estimate-section}
\|s\|_h\leq C(C_s,n)l^{\frac{n}{2}}(\int_M|s|^2 {\rm dv}_g)^{\frac{1}{2}}.\end{align}
On the other hand, we have the following Bochner formula,
$$\Delta|\nabla s|_h^2=|\nabla\nabla s|^2+|\bar{\nabla}\nabla s|^2-(n+2)l|\nabla s|^2+\langle\text{Ric}(\nabla s,.),\nabla s\rangle.$$
 Then we  can also apply the  Moser iteration  to obtain a $L^\infty$-estimate for $|\nabla s|_h^2$  as done for $|s|_h^2$.
In fact, it  suffices to deal with  the extra integral terms  like $\langle\text{Ric}(\nabla s,.),\nabla s\rangle|\nabla s|^{2p}$.
  But those terms  can be controlled by the integral of $(|\nabla\nabla s|^2+|\bar{\nabla}\nabla s|^2)|\nabla s|_h^{2p}$
by taking  integral  by parts with the help of the condition (\ref{Ricci-potential-gradient}) (cf. \cite{WZ2}, \cite{TZZZ}).
 As a consequence,   we obtain
\begin{align}\label{gradient-estimate-section}
\|\nabla s\|_{h}\leq C(C_s,n)l^{\frac{n}{2}}(\int_M|\nabla s|^2 {\rm dv}_g)^{\frac{1}{2}}\leq C(C_s,n)l^{\frac{n+1}{2}}(\int_M|s|^2 {\rm dv}_g)^{\frac{1}{2}}.
\end{align}
Therefore,  combining (\ref{l^2-estimate-section}) and (\ref{gradient-estimate-section}),  we derive (\ref{section-estiamte-1}).

\end{proof}
\begin{rem}\label{remark-1}
Using the same argument in Lemma \ref{section-gradient-estimate},  we can prove:  If $(M,g)$ satisfies
$${\rm Ric}(\omega_g)\geq -\Lambda^2\omega_ g+\sqrt{-1}\partial\bar{\partial}u,$$
for some $u$ with $|\nabla u|_g\le A$, then
\begin{align}\label{section-estiamte-1}
\|s\|_h+l^{-\frac{1}{2}}\|\nabla s\|_{h}\leq C(C_s,A,\Lambda)l^{\frac{n}{2}}(\int_M|s|^2 {\rm dv}_g)^{\frac{1}{2}},  ~\forall s\in H^0(M,L^l).
\end{align}
\end{rem}

\begin{lem}\label{gradient-section}
Let $(M,g)$ be  a Fano manifold  which satisfies  (\ref{Ricci-potential-gradient}) as in Lemma \ref{section-gradient-estimate}.
 Let  $\bar{\partial}$-operator  be defined for smooth sections  on  $(M, L^l)$  ( $l\geq 4n $ ) with the induced metric $h$.
 Then for any  $\sigma\in C^\infty(\Gamma(M,L^l))$,  there exists a solution $v\in C^\infty(\Gamma(M,L^l))$ such that
$\bar{\partial}v=\bar{\partial}\sigma$ with property:
\begin{align}\label{L^2}
\int_M|v|^2\leq 4l^{-1}\int_M|\bar{\partial}\sigma|^2.
\end{align}
\end{lem}

\begin{proof}  The existence part comes from  the H\"omander $L^2$-theory. We suffice to  verify   (\ref{L^2}), which is equal to prove
 that the first eigenvalue  $\lambda_1(\bar\partial, L^l)$  of $\Delta_{\bar{\partial}}$  is greater than $\frac{l}{4}$,
 where  $\Delta_{\bar{\partial}}$  denotes  the  Lapalce operator defined  on $L^2(T^*M\bigotimes L^l)$.

Note that  the following two  identities hold for any $\theta\in \Omega^{0,1}(L^l)$,
$$\Delta_{\bar{\partial}}\theta=\bar{\nabla}^*\bar{\nabla}\theta+\text{Ric}(\theta,.)+l\theta$$
and
$$\Delta_{\bar{\partial}}\theta=\nabla^*\nabla \theta-(n-1)l\theta.$$
It follows
\begin{align}
\Delta_{\bar{\partial}}\theta=(1-\frac{1}{2n})\bar{\nabla}^*\bar{\nabla}+(1-\frac{1}{2n})\text{Ric}(\theta,.)+
\frac{1}{2n}\nabla^*\nabla \theta+\frac{l}{2}\theta.
\end{align}
Then with the help of  condition  (\ref{Ricci-potential-gradient}),  a direct computation shows
\begin{align}\label{partial-operator-eingenvalue}
&\int_M\langle \Delta_{\bar{\partial}}\theta,\theta\rangle\notag\\
&=(1-\frac{1}{2n})\int_M|\bar{\nabla}\theta|^2+\frac{1}{2n}\int_M|\nabla\theta|^2
+\frac{l}{2}\int_M|\theta|^2&\notag\\
&+(1-\frac{1}{2n})\int_M(|\theta|^2+\langle \nabla \bar{\nabla}u(\theta,.),\theta\rangle)\notag\\
&\geq (1-\frac{1}{2n})\int_M|\bar{\nabla}\theta|^2+\frac{1}{2n}\int_M|\nabla\theta|^2+\frac{l}{2}\int_M|\theta|^2\notag\\
&+(1-\frac{1}{2n})\int_M|\theta|^2-(1-\frac{1}{2n})\int_M \langle \bar{\nabla}u, (\langle \nabla\theta,\theta\rangle
+\langle \theta,\bar{\nabla}\theta\rangle)\rangle&\notag\\
&\geq (1-\frac{1}{2n})\int_M|\bar{\nabla}\theta|^2+\frac{1}{2n}\int_M|\nabla\theta|^2
+\frac{l}{2}\int_M|\theta|^2\notag\\
&+(1-\frac{1}{2n})\int_M|\theta|^2-(1-\frac{1}{2n})\int_M[\frac{1}{2n}(|\bar{\nabla}\theta|^2+|\nabla\theta|^2)+n|\theta|^2]&\notag\\
&\geq (\frac{l}{2}-n)\int_M|\theta|^2.
\end{align}
Now we can choose $l\ge 4n$ to get  that   $\lambda_1(\bar\partial, L)\ge\frac{l}{4}$  as required.
\end{proof}

\begin{rem}\label{remark-lemma-3.2}
If  the upper bound  of $|\nabla u|$ is  replaced   by  a constant $C$,   the coefficient  at  the last inequality in (\ref{partial-operator-eingenvalue}) will be $\frac{l}{2}-nC^2$ .
Then by choosing  $l\geq 4nC^2$,  one  can also get (\ref{L^2}).   This  was  proved in \cite{TZha}.
\end{rem}

Recall  that a sequence of  almost K\"{a}hler-Einstein Fano manifolds $(M_i, J_i,$ $ g^i)$  satisfy:
\begin{align}\label{almost-ke-condition}
&i) ~{\rm Ric}(g^i)\geq -\Lambda^2 g^i~\text{and} ~{\rm diam}(M_i, g^i)\leq D;\notag\\
&ii)~\int_{M_i}|{\rm Ric} (g^i)-g^i| \text{{\rm{dv}}}_{g^i}\rightarrow 0;\notag\\
 &iii) ~\int_0^1\int_{M_i}|R(g^i_t)-n|\text{{\rm{dv}}}_{g^i_t} dt \rightarrow 0,~{\rm as}~i\to\infty.
\end{align}
 Here $g^i$ are normalized so that  $\omega_{g^i}\in 2\pi c_1(M_i)$ and   $g^i_t$ are  the solutions of (\ref{Ricci flow}) with the initial metrics $g^i$. We note that  $\text{vol}(M_i,g^i)=(2\pi)^n c_1(M_i)^n\ge V$ for some uniform constant $V$ by the normalization.

Applying  Lemma \ref{section-gradient-estimate} and  Lemma \ref{gradient-section} to almost K\"{a}hler-Einstein manifolds with the help of
  gradient estimate (\ref{gradient-estimate-2}) in Proposition \ref{estimate-u-R},   we have the following proposition.

\begin{prop}\label{corollary-1}
 Let  $\{(M_i,g^i)\}$  be a sequence of  almost K\"{a}hler Einstein metrics which satisfy
  (\ref{almost-ke-condition}).
 Then  for any $t\in (0, 1)$  there exists an integer
  $N=N(t)$   such that for any $i\ge N$ and $l\ge  4n$   it holds,
 \begin{align}\label{uniform-gradient}
\|s\|_{h_t^i}+l^{-\frac{1}{2}}\|\nabla s\|_{h_t^i}\leq C l^{\frac{n}{2}}(\int_M|s|^2  \text{{\rm{dv}}}_{g_t^i})^{\frac{1}{2}}
\end{align}
and
\begin{align}\label{L^2-ke}
\int_{M_i}|v|_{h_t^i}^2\leq 4l^{-1}\int_{M_i}|\bar{\partial}\sigma|^2.
\end{align}
 Here   $s\in H^0(M_i, K_{M_i}^{-l})$,   the norms of $|\cdot|_{h^i_t}$  are induced by $g_t^i$, and  $C$ is a uniform constant independent of $t$.
 \end{prop}

 \begin{proof}  A well-known result  shows  that  the Sobolev constants  $C_s$ of
 $(M_i,$ $ g^i)$  depend only on the constants  $\Lambda,D$ and $V$. Then by  (\ref{gradient-estimate-2}) in Proposition \ref{estimate-u-R},  for any $t\in (0,1)$,  there exists $N=N(t)$ such that
$$\|\nabla u^i\|_{h^i_t}\leq 1,~ \forall~i\ge N,$$
where $u^i$ are Ricci potentials of $g^i_t$ .  Thus we can apply   Lemma \ref{section-gradient-estimate} to get (\ref{uniform-gradient}).
Similarly, we can get (\ref{L^2-ke}) by Lemma \ref{gradient-section}.
\end{proof}

\section{Construction of locally approximate holomorphic sections }

Let $\{(M_i,g^i)\}$ be a sequence of  almost K\"{a}hler-Einstein manifolds as in Section 3 and $(M_\infty,g_\infty)$  its   Gromov-Hausdorff limit.
It was proved by Tian and  B. Wang  that the regular part $\mathcal{R}$ of $M_\infty$ is an open K\"ahler manifold
and the codimension of singularities of $M_\infty$ is at least 4 \cite{TW}.   Moreover, according to Proposition 5.1 in that paper, we have

\begin{lem}\label{regular-part} Let $x\in M_\infty$.
Then there exist constants $\epsilon=\epsilon(n)$ and  $r_0=r_0(n, C)$ such that if ${\rm vol}(B_x(r))\geq(1-\epsilon)\omega_{2n}r^{2n}$ for some $r\leq r_0$, then
$B_x(\frac{r}{2})\subseteq\mathcal{R}$, and
\begin{align}
{\rm Ric}(g_\infty)=g_\infty, \, \|\nabla^l {\rm Rm}\|_{C^0(B_x(\frac{r}{2}))}\leq \frac{C}{r^{l+2}},\notag
\end{align}
 where  the constant $C$ depends only on $l$,  and the constants $\Lambda$ and $D$ in (\ref{almost-ke-condition}).
\end{lem}

Recall that  a tangent cone $C_x$ at $x\in M_\infty$ is a  Gromov-Hausdorff limit defined by
\begin{align}\label{tagent-cone-sequence}(C_x, g_x, x)=\lim_{j\rightarrow \infty}(M_\infty,\frac{g_\infty}{r_j^2},x),\end{align}
where $\{r_j\}$ is some sequence which goes to $0$.  Without  the loss of  generality,  we  may  assume that $l_j=\frac{1}{r_j^2}$ are integers.   Since $(C_x, g_x, x)$ is a metric cone, $g_x={\rm hess}\frac{\rho_x^2}{2}$,   where $\rho_x={\rm dist}(x,\cdot)$ is a distance function staring  from $x$ in $C_x$.

Denote   the regular part of $(C_x, g_x, x)$ by  $\mathcal{CR}$,   which consists of points in $C_x$ with flat cones.
By Lemma  \ref{regular-part}, we prove

\begin{lem}\label{cone-regular-part}  $\mathcal{CR}$ is an open K\"ahler-Ricci flat manifold.  Moreover, for any compact set $K\subset \mathcal{CR}$,  there exist  a  sequence of  $(K_j \subset  \mathcal{R},  \frac{1}{r_j^2} g_\infty)$  which converges to  $K$ in  $C^\infty$-topology.
\end{lem}

\begin{proof} Let $\epsilon$ be a small number chosen as in Lemma  \ref{regular-part}.  Then for any  $y\in\mathcal{CR}$,
there exists some small $r$ such that $\hat B_y(r)\subset C_x$ and
\begin{align}
\text{vol }(\hat B_y(r))\geq(1-\frac{\epsilon}{2})\omega_{2n}r^{2n}.\notag
\end{align}
Thus  there exists  a sequence of  $y_\alpha\in C_x$  such that
\begin{align}
\text{vol }(B_{y_\alpha}(rr_\alpha))\geq(1-\epsilon)\omega_{2n}(rr_\alpha)^{2n},\notag
\end{align}
where  the sequence $\{r_\alpha\}$ is chosen as in  (\ref{tagent-cone-sequence}).
By Lemma \ref{regular-part},   it follows
\begin{align}
\|{\rm Rm}(\tilde{g}_\infty)\|_{C^l(\tilde B_{y_\alpha}(\frac{r}{2}))}\leq \frac{C_l}{r^{l+2}},\notag
\end{align}
 where $\tilde{g}_\infty=\frac{g_\infty}{r_\alpha^2}$ and $\tilde B_{y_\alpha}(\frac{r}{2})\subset M_\infty$  is a $\frac{r}{2}$-geodesic ball  with respect to the metric  $\tilde{g}_\infty$.
Hence, by the Cheeger-Gromov compactness theorem [GW],  $ (\tilde B_{y_\alpha}(\frac{r}{2}), \tilde{g}_\infty)$   converge to $(\hat B_{y}(\frac{r}{2}), g_x)$  in  $C^\infty$-topology.
In particular,  $ B_{y_\alpha}(\frac{r_\alpha r}{2}$
\newline  $)\subset\mathcal{R} $ and $\hat B_{y}(\frac{r}{2})\subset \mathcal{CR}$.  This implies that  $\mathcal{CR}$ is an open manifolds.
Moreover,  $\mathcal{CR}$  is a K\"ahler-Ricci flat manifold  since  each $ ( B_{y_\alpha}(\frac{r_\alpha r}{2}), {g}_\infty)$ is an open K\"ahler-Einstein manifold.
If $K $ is a compact set  of $\mathcal{CR}$, then
by taking finite  small geodesic covering balls,  one can  find  a sequence  $\{(K_j\subset  \mathcal{R},  \frac{1}{r_j^2} g_\infty)\}$  which converges to $(K, g_x)$ in  $C^\infty$-topology.
\end{proof}

Define an open  set  $V(x;\delta)$  of $\mathcal{CR}$  by
 \begin{align}
 V(x;\delta)=\{y\in C_x|~{\rm dist}(y,S_x)\geq \delta, d(y,x)\leq\frac{1}{\delta}\},
 \end{align}
where $S_x=C_x\setminus  \mathcal{CR}$.  The following lemma shows  that there  exists a   ``nice'' cut-off function on $C_x$ which  supported on   $V(x;\delta)$.

\begin{lem}\label{cut-off}
For any $\eta, \delta>0$, there exist some $\delta_1<\delta$ and a cut-off function $\beta$  on $C_x$ which supported in $V(x;\delta_1)$ with property:
$\beta=1, \text{ in } V(x;\delta);$
\begin{align}\int_{C_x}|\nabla \beta|^2 e^{-\frac{\rho_x^2}{2}} d{\rm v}_{g_x}\leq \eta.\notag
\end{align}
\end{lem}

Lemma \ref{cut-off} is  in fact a corollary of following fundamental lemma.

\begin{lem}\label{cut-off-0}
Let $(X^m,d,\mu)$ be a  measured metric space  such that
\begin{align}\label{volume}
 \mu(B_y(r)\le C_0r^m,  ~\forall ~r\leq 1, ~y\in X.
\end{align}
 Let  $Z$ be a  closed subset of $X$  with $\mathcal{H}^{m-2}(Z)=0$. Suppose that  there exists a nonnegative function $f\leq 1$ on $X$ such that
$$\int_Xfd\mu\leq 1.$$
 Then
for any $x\in X$,  $\eta>0$ and $\delta>0$, there exist  a  positive $\delta_1\leq\delta$  and  a cut-off function $\beta\ge 0$, which  supported in $B_x(\frac{1}{\delta_1})\setminus Z_{\delta_1}$ with property: $
{\beta}=1, \text{ in } B_x(\frac{1}{\delta})\setminus Z_{\delta};$
\begin{align}\label{small-gradient}\int_Xf|{\rm Lif} (\beta)|^2d\mu \leq \eta.
\end{align}
Here $Z_{\delta_1}=\{x'\in X|~ {\rm dist}(x',Z)\leq \delta_1\}$ and ${\rm Lip} (\beta)(z)=\sup_{w\rightarrow z}|\frac{f(w)-f(z)}{d(w,z)}|$.

\end{lem}

\begin{proof}
Let  $R\geq \sqrt{\frac{8}{\eta}}+ \frac{2}{\delta}$.   Since $\mathcal{H}^{m-2}(Z)=0$,  then for any $\kappa>0$,
we can  take  finite geodesic  balls  $B_{x_i}(r_i)$  $(r_i\leq \delta)$  with  $x_i\in Z$ to cover  $B_x(R)\bigcap Z$  such that
\begin{align}
\Sigma_ir_i^{m-2}\leq \kappa.\notag
\end{align}
Let $\zeta:\mathbb{R}\rightarrow \mathbb{R}$  be  a cut-off function  which satisfies:
 \begin{align}
\zeta(t)=1, \text{ for }t\leq \frac{1}{2};   \zeta(t)=0, \text{ for }t\geq 1; |\zeta'(t)|\leq2.\notag
\end{align}
Set
$$
\chi(y)=\min_i\{ 1-\zeta(\frac{d(y,x_i)}{r_i})\}$$
and
$$ \beta(y)=\zeta(\frac{\epsilon}{d(y,x)})\zeta(\frac{d(y,x)}{R})\chi(y),
$$
where $\epsilon\le \frac{\delta}{2}$.   Then  it is easy to see that  $\beta$ is supported in  $B_x(R)\setminus \cup B_{x_i}(\frac{r_i}{2})$ with $\beta\equiv 1$ in  $B_x(\frac{1}{\delta})\setminus Z_{\delta}$.  Moreover,
\begin{align}
\int_Xf|Lif \beta|^2d\mu& \leq 4C_0\Sigma_ir_i^{-2}r_i^{m}+ 4C_0\epsilon^{m-2}+\frac{4}{R^2}\notag\\
& \leq 4C_0\kappa+ 4C_0\epsilon^{2n-2}+\frac{\eta}{2}.\notag
\end{align}
 Thus, if we  choose  $\epsilon$ and  $\kappa$  such  that
$4C_0\kappa+ 4C_0\epsilon^{2n-2}\leq \frac{\eta}{2}$,
then  we get (\ref{small-gradient}).
By choosing   $\delta_1 \leq\min\{\frac{\epsilon}{2}, \frac{1}{2R}\}$ such that
$$Z_{\delta_1}\cap B_x(R)\subseteq\cup B_{x_i}(\frac{r_i}{2}),$$    we  can also get
  ${\rm supp}(\beta)\subset B_x(\frac{1}{\delta_1})\setminus Z_{\delta_1}.$
Hence  $\beta$  satisfies all conditions  required in  the lemma.

\end{proof}

\begin{proof}[Proof of Lemma \ref{cut-off}]
Applying Lemma \ref{cut-off-0} to $X=C_x, Z=S_x, f=e^{-\frac{\rho_x^2}{2}}$, we get the lemma.
\end{proof}

By Lemma \ref{cone-regular-part},  we see that  for any $\delta>0$  there  exists a sequence of  $K_j\subset (M_\infty, r_j^{-2}g_\infty) $ which converge
to $V(x;\delta)$. Let $L_0=(C_x, \mathbb{C})$  be  the trivial holomorphic bundle  over $C_x$  with a hermitian metric  $h_0=e^{-\frac{\rho_x^2}{2}}$.
 Then $h_0$ induces  the  Chern connection $\nabla_0$ with its curvature
 $${\rm Ric}(L_0,\nabla_0)=g_x.$$
  In the following we  show that   a sufficiently  large multiple line bundles of   $K^{-1}
 _{\mathcal{R}}|_{K_j}$ will  approximate   to $L_0$ over $V(x;\delta)$.  This is in fact an application  of the following fundamental lemma.

\begin{lem}\label{fundamental-lemma-2}
Let   $(V,g)$  be a $C^2$ open Riemannian manifold and   $U, U'\subset \subset V$  are two  pre-compact open subsets of $V$ with $ U\subset \subset  U'$.
Then   for any positive number $\epsilon$, there exist  a small number $\delta=\delta(U', g, \epsilon)$  and a  positive integer $N=N(U,g, \epsilon)$, which  depends on the fundamental group of $U$, the metric $g$ on $U$, and the small $\epsilon$  such that the following is true:
if a hermitian complex line bundle $(L,h)$ over $V$ with  associated  connection $\nabla$ satisfies
\begin{align}\label{flat-curvature-bundle}
|{\rm Ric}^\nabla|_{g}\leq \delta, ~{\rm  in }~ U',
\end{align}
then there exist  a positive integer  $l\leq N$ and   a section $\psi$ of $L^{\otimes l}$  over $U$ with  $|\psi |_{h}\equiv 1$  which satisfies
\begin{align}\label{psi-nabla}
|D^{\nabla \otimes l} \psi|_{h,g}\leq \epsilon, ~{\rm in} ~ U.
\end{align}

\end{lem}

\begin{proof} The proof seems standard.
First  we show that $(L, U)$ is a flat bundle with respect to some connection.
 Let  $B_{x_i}(r_i)$ ($r_i\leq 1$)  be finite convex geodesic balls in $V$
 such that $\bar{U}\subset \cup B_{x_i}(r_i)\subset U' $.    Then for $y \in B_{x_i}(r_i)$  there exists a  minimal geodesic curve $\gamma_y$ in $B_{x_i}(r_i)$,
   which connects $x_i$ and  $y$.  Picking  any  vector $s_i\in L_{x_i}$ with $|s_i|=1$  and using   the parallel transportation,
 we   define  a  parallel vector  field by
\begin{align}
e_i(y)={\rm Para}_{\gamma_y}(s_i), ~\forall~y\in B_{x_i}(r_i).\nonumber
\end{align}
In particular,  $De_i(x_i)=0$.  Let  $T$  be a  vector field,   which is tangent to  $\gamma_y$,  and   $X$  another vector field with  $[T,X]=0$.  Then
\begin{align}\label{ricci-formula-bundle}
D_T[D_Xe_i]=D_X[D_Te_i]+{\rm Ric}^\nabla (T,X)e_i={\rm Ric}^\nabla (T,X)e_i.
\end{align}
By the condition (\ref{flat-curvature-bundle}),    it  follows
\begin{align}\label{gradient-vector}
|D e_i|_{h,g}\leq C({U'},g)||\text{Ric}^\nabla||_{({U'},g)}\le  C({U'},g)\delta,~{\rm in}~ B_{x_i}(r_i) .
\end{align}
This implies that  the transformation function  $g_{ij}$ of $L$ is  nearly  constant in $B_{x_i}\cap B_{x_j}$.
Since
 the first Chern class   lies in the secondary integral cohomology group,   $L$ is topologically  trivial  as long as  $\delta$ is small, i.e., $c_1(L)=0$.
 Hence,   there exist
 some complex functions $f_i$ over $B_{x_i}(r_i)$ such that
\begin{align}\label{gradient-f}
|D f_i|\leq  C({U'},g) \delta<<1,
\end{align}
and the transition functions for $\tilde{e_i}=f_ie_i$ are constant.  Here  $\|\tilde{e_i}\|_{h}=\|e_i\|_h.$   As a consequence,   we can define an associated connection $\nabla'$ on $L$ to $h$ such that
$$|\tilde{e_i}|_{h}=1~{\rm and}~D^{\nabla'}\tilde{e_i}=0.$$
 In fact,  if we   set  $\nabla'=\nabla+\alpha\otimes e_i $,   then locally,
$$ D^{ \nabla'}\tilde{e_i}=D^\nabla(f_ie_i)+\alpha(\tilde{e_i})=f_iD^\nabla e_i+df_i\otimes e_i+ f_i\alpha\otimes e_i.$$
Thus
$$\alpha=-\frac{1}{f_i}(df_i +\langle  f_iD^\nabla e_i, e_i\rangle_h) $$
 which  is   uniquely determined by  requiring $ D^{\nabla'}(\tilde{e_i})=0$.
 Therefore, $ (L, \nabla')$ is a  flat bundle over $U$  with respect to $\nabla'$.  Moreover,  by (\ref{gradient-vector}) and (\ref{gradient-f}), we have
\begin{align}\label{nabla-comparison}   \| \nabla-\nabla' \|_{(U,g)} = \|\alpha \|_{(U,g)}\leq  C({U'},g)||\text{Ric}^\nabla||_{({U'},g)}\le C({U'},g)\delta.\end{align}

 Next we note that
the holonomy group  of a  flat bundle over $U$  is an element of $Hom(\pi_1(U),\mathbb{S}^1)\cong G\times \mathbb{T}^k$ for some finite group $G$ with  order $m_1$,
where $k$ is the  Betti number of $\pi_1(U)$.   By  the pigeon-hole principle,  we see  that for any $\gamma$-neighborhood $W\subseteq\mathbb{T}^k$ of  the identity there
exists a positive integer $m_2=m_2(\gamma)$  such that for any element $\rho\in\mathbb{T}^k$, $\rho^a\in W$ for some  number $a $ $(1\leq a\leq m_2).$
As a consequence,
for any  element  $t\in  G\times \mathbb{T}^k$, there exists
$l~ (1\leq l\leq N=m_1m_2)$  such that   $t^l\in W.$
Hence,  there exist  $l$  and a smooth section $\psi$ of $L^{\otimes l}$ over $U$ by perturbing a parallel vector field in $L^{\otimes l}$
   such that
   $$||\psi|_h-1|, \| \psi^{{\nabla'}^{\otimes l}}\|_{h,g}\leq  C(U',g)\gamma(\delta), {\rm in}~ U. $$
Moreover, By (\ref{nabla-comparison}),  we  can   normalize $\psi$ by $|\psi|_h\equiv 1$ so that $(\ref{psi-nabla})$ is true. The lemma is proved.
\end{proof}

\begin{prop}\label{almost-falt-bundle}  Let $x\in M_\infty$ and $\delta_1>0$. Then for any $\epsilon>0$, there exist a positive
 integer $N=N(V(x;\delta_1),\epsilon)$  and a  large integer   $j_0$ such that  for  $j\ge j_0$  there exist $l=l(j)\le N$,  and
 a  sequence of $ {K}_j\subseteq M_\infty$ and   a sequence of   pairs of isomorphisms ($\phi_j,\psi_j$)
with property:
 \begin{align}\label{bundle}
\CD
 L_0@>{\psi}_j>> K_{\mathcal{R}}^{-ll_j}|_{{K}_j} \\
  @V  VV @V  VV  \\
  V(x;\delta_1) @>{\phi}_j>>{K}_j,
\endCD
\end{align}
which satisfy
\begin{align}
{\phi}_j^*(ll_j g_\infty)\rightarrow g_x,~{\rm as}~j\to\infty,\notag
\end{align}
and
$$ | D \psi_j|_{g_x}\le \epsilon,   ~{\rm in}~V(x;\delta_1).$$
\end{prop}

\begin{proof}

Define an open  set  $U$  of $\mathcal{CR}$  by
 \begin{align}
U= U(x;\epsilon_1,\epsilon_2,R)=\{y\in C_x|~{\rm dist}(\bar{y},S_x)\geq \epsilon_1, \epsilon_2 \leq d(y,x)\leq R\},\notag
 \end{align}
 where $\bar{y}$ is the projection to  the section $Y$ of $C_x=C(Y)$.  Then there exist some $\epsilon_1,\epsilon_2$ and $R$  such that
$$V(x;\delta_1)\subseteq  U(x;\epsilon_1,\epsilon_2,R). $$
Moreover, we can choose a sequence of integers  $l_j=\frac{1}{ r_j^2}$   such that
$$(M_\infty,  l_j g_\infty, x)\rightarrow (C_x, g_x, x),~{\rm as}~j\to \infty.$$
Hence by  Lemma  \ref{cone-regular-part}, there exist a sequence of   $ \tilde {K}_j\subseteq M_\infty$ and   a sequence of  diffeomorphisms $\tilde{\phi}_j$ from
$ U(x;\epsilon_1,\frac{\epsilon_2}{\sqrt{N}},R)$ to  $\tilde K_j$ such that $\tilde{\phi}_j^*(l_j g_\infty)\rightarrow g_x,$ where $N=N(U,g_x,\epsilon)$ is a large integer
as determined in   Lemma \ref{fundamental-lemma-2}.

Let  $h_\infty$ be  the induced hermitian metric on $K_{\mathcal R}^{-1}$  by $g_\infty$ on the regular part $\mathcal R$ of
$M_\infty$. Let
$$ (L_j, h)=\tilde{\phi}_j^* (K_{\mathcal R}^{-l_j},h_\infty^{\otimes l_j})\otimes (L_0, h_0)^{*}$$
be product complex line bundle on $U$, where  $h$ is an induced  hermitian metric by $h_\infty$ and $h_0$  with  associated  connection $\nabla_j$ on $L_j$ for each $j$.
Clearly,
  $$\|\text{Ric}^{\nabla_j}\|_{(U',g_x)}\leq \delta<<1,$$
as long as $j$ is large enough, where     $U'\subset\subset \mathcal CR$ is an open set such that $\bar U\subset\subset U'$.
 Applying  Lemma \ref{fundamental-lemma-2} to $L_j$
over  $U'$,  we  see that
there exist  some positive integer $l=l(j)\leq N$  and a section $\psi'$ on  $L_j^{\otimes l}$ such that
$$|D^{\nabla_j^{\otimes l}}\psi'|_{(U,g_x)}\leq \epsilon.$$

Let $Y_{\epsilon_1}= U(x;\epsilon_1,\epsilon_2,R)\bigcap Y$ and $\tilde{\psi}$ an
  extension section over  $U(x;\epsilon_1,\frac{\epsilon_2}{\sqrt{l}},$
  \newline $R)$ of  the restriction of  $\psi'$ on $Y_{\epsilon_1}$
    by  the parallel transportation along rays from $x$.  Clearly,
    $$\|\tilde \psi\|_{\otimes^l h}\equiv 1.$$
     Moreover,  by the formula (\ref{ricci-formula-bundle}), it is easy to see
\begin{align}
|D^{\nabla_j^{\otimes l}}\tilde{\psi}|_{(U(x;\epsilon_1,\frac{\epsilon_2}{\sqrt{l}},R),g_x)}\leq \frac{\sqrt{l}}{\epsilon_2}(\epsilon+C_0R^2\delta),
\end{align}
where  the constant $C_0$ depends only on $(Y,g_x)$. Thus we  have  pairs of isomorphisms $(\tilde{\phi}_j, \tilde{\psi}_j)$ with property:
\begin{align}\label{bundle-1}
\CD
  L_0^l@>\tilde{\psi}_j>> K_{M_\infty}^{-l_jl}|_{K_j} \\
  @V  VV @V  VV  \\
(U(x;\epsilon_1,\frac{\epsilon_2}{\sqrt{l}},R),g_x)  @>\tilde{\phi}_j>>(K_j,l_jg_\infty),
\endCD
\end{align}
which satisfy
\begin{align}\label{no-scale}
|D \tilde{\psi}_j|_{g_x}\le 2\frac{\sqrt{l}}{\epsilon_2}\epsilon,
\end{align}
 as long as  $j$ is large enough.

Rescaling $U(x;\epsilon_1,\epsilon_2,R)$  into  $U(x;\epsilon_1,\frac{\epsilon_2}{\sqrt{l}},R)$ by
 $$\mu_l:~y\rightarrow \frac{y}{\sqrt{l}}, y\in U(x;\epsilon_1,\epsilon_2,R). $$
 We have isometrics
  $$\mu_l^*L_0^l\cong L_0,\mu_l^*g_x=\frac{g_x}{l}.$$
By (\ref{bundle-1}),   it follows
\begin{align}\label{bundle-2}
\CD
L_0@>\tilde\psi_j\circ (\mu_l^*)^{-1} >> K_{M_\infty}^{-l_jl}|_{K_j} \\
  @V  VV @V  VV  \\
(U(x;\epsilon_1,\epsilon_2,R),\frac{g_x}{l}) @>\tilde\phi_j\circ \mu_l >>(K_j,l_jg_\infty).
\endCD
\end{align}
 Let
  $${\phi}_j=\tilde\phi_j\circ \mu_l, ~{\rm and}~{\psi}_j=\tilde\psi_j\circ (\mu_l^*)^{-1}.$$
  Note that $V(x;\delta_1)\subseteq U(x;\epsilon_1,\epsilon_2,R).$    Then $K_j={\phi}_j(V(x;\delta_1))$ is well-defined.
 Hence,  rescaling the metric $\frac{g_x}{l}$  back to $g_x$, we get from (\ref{no-scale}),
 \begin{align}\label{gradient-psi-2} |D {\psi}_j|_{g_x}\le 2\frac{\epsilon}{\epsilon_2}, ~{\rm in}~V(x;\delta_1).\end{align}
Replacing $2\frac{\epsilon}{\epsilon_2}$ by $\epsilon$, we prove the proposition.
\end{proof}

Proposition \ref{almost-falt-bundle} will be used to construct  peak sections of  holomorphic line bundles over
  a sequence of K\"ahler manifolds in next section.

\vskip3mm

\section{$\bar\partial$-equation  and construction of  holomorphic sections }

In this section, we give a construction of  peak  holomorphic sections  by solving  $\bar\partial$-equation   on a  smoothing
sequence of almost K\"ahler-Einstein manifolds in \cite{TW}.  We will use the rescaling method as  done  for the K\"ahler-Einstein manifolds sequence in \cite{DS},  \cite{T6}.

\begin{prop}\label{partial}
Let $\{(M_i, g^i)\}$ be  a sequence of almost K\"ahler-Einstein  Fano manifolds as in Section 3 and $ (M_\infty, g_\infty)$ be its  Gromov-Hasusdorff limit. Then  for any
sequence  of  $p_i\in M_i$ which converges  to $x\in M_\infty$,  there  exist  two large number $l_x$ and $i_0$,  and
  a small time  $ t_x$    such that for any $i\ge i_0$  there exists a holomorphic section  $s_i\in \Gamma (K_{M_i}^{-l_x}, h^i_{t_x})$ which satisfies
 \begin{align}\label{l2-norm-less-1}
 \int_{M_i}|s_i|_{h^i_{t_x}}^2{\rm dv}_{g^i_{t_x}}\le 1~{\rm and}~
 |s_i|_{h^i_{t_x}}(p_i) \geq  \frac{1}{8},
\end{align}
 where $g^i_t$   are  solutions  of (\ref{Ricci flow}) with the initial metrics  $g^i$ and  $h^i_{t_x}$  are  the hermitian metrics of $K_{M_i}^{-l_x}$  induced by
$g^i_{t_x}$.
\end{prop}

\begin{proof}
As in Section 4,  we let
\begin{align}
(C_x,\omega_x,x)=\lim_{j\rightarrow \infty}(M_\infty,\frac{g_\infty}{r_j^2},x).\notag
\end{align}
 Choose a $\delta$ so that
$\delta \leq (2\pi)^{-\frac{n}{2}}\frac{ C_1}{64}$,   where $C_1$ is a constant  chosen as in (\ref{uniform-gradient}).  We  consider the $\overline\partial$-equation  for sections
 on the  trivial  line bundle
$L_0=(V(x;\delta), \mathbb C )$,
  \begin{align}
  \bar{\partial}\sigma=f, ~\forall ~ f\in \Gamma^\infty( (TV^*)^{(0,1)}\otimes  L_0).\notag
  \end{align}
 Then the standard $C^0$-estimate  for the elliptic equation shows
\begin{align}\label{elliptic}
  |\sigma|_{C^0(V(x;2\delta))}\leq C_2(|f|_{C^0(V(x;\delta))}+\delta^{-n}[\int_{V(x;\delta)}|\sigma|^2 d{\rm v}g_x]^\frac{1}{2}),
  \end{align}
where the constant  $C_2$ depends on the metric $g_x$.

 Let  $0<\eta\le \frac{\delta^{2n}}{1000C_2^2}$ and $\beta$  a cut-off function  supported in $V(x;\delta_1)$  constructed in  Lemma \ref{cut-off}.
Let $K_j$ be the  sequence of open sets in $M_\infty$ which converge to $V(x;\delta_1)$  and  $\psi_j$ be  the  sequence of  isomorphisms  from $L_0$ to
 $ K_{ \mathcal{R}}^{-ll_j}|_{{K}_j} $
 constructed  in Proposition  \ref{almost-falt-bundle}, where $l=l(l_j)\le N=N(V(x;\delta_1),\epsilon) $ and $l_j=\frac{1}{r_j^2}$.  Set  $\tau_j=\psi_j(\beta e)$,  where $e$ is  a  unit basis of
$L_0$. Then  $\{\tau_j\}$ is a  sequence of smooth sections of $K_{\mathcal{R}}^{-l_jl}$  supported in $\psi_j(V(x;\delta_1))$.   Moreover,  $\tau_j$ satisfies
the following property as long as $j$ is large enough:
\begin{align}\label{l-norm}
&i)~ \|\tau_j\|^2_{C^0(\phi_i(V(x;\delta)\bigcap B_x(3\delta)))}\geq \frac{3}{4}e^{-3\delta^2}\geq \frac{1}{2};\notag\\
&ii)~\int_{M_\infty}|\tau_j|^2{\rm dv}_{g_\infty}\leq \frac{3}{2}\frac{r_j^{2n}}{l^n}(2\pi)^n;\notag\\
&iii)~\bar{\partial}_{J\infty}\tau_j\leq\frac{\eta}{8},~{\rm in}~V(x;\delta);\notag\\
&iv)~  \int_{M_\infty}|\bar{\partial}_{J_\infty}\tau_j|^2 {\rm dv}_{g_\infty}\leq \frac{3}{2}r_j^{2n-2}\frac{\eta}{l^{n-1}}.
\end{align}

On the other hand,  from the proof of Lemma \ref{cone-regular-part}, we see that
there exists  $t_0$, which depends  on $V(x;\delta_1)$  such that for any sufficiently large $j$ it holds
$${\rm vol}(B_y( \sqrt{t_0}\frac{r_j}{\sqrt{l}})\ge (1-\epsilon){\rm vol}(B_0( \sqrt{t_0}\frac{r_j}{\sqrt{l}}),~\forall ~y\in  K_j,$$
where $\epsilon$ is a small  constant chosen as in Lemma \ref{regular-part}.  Then by the pseudo-locality theorem in \cite{TW},   there exist a $t_0'\le t_0$,  and
  a sequence  of sets $B_{i}\subseteq  M_{i}$  and a sequence of  diffeomorphisms $\varphi_i:  {K}_j\rightarrow B_{i}$ such that
\begin{align}
\varphi_{i}^*g^{i} (t_0'\frac{r_j^2}{l})\rightarrow g_\infty,\notag\\
\varphi_i^* J_i\rightarrow J_\infty,\notag\\
\varphi_{i}^*K_{M_{i}}^{-1}\rightarrow K_{\mathcal{CR}}^{-1},\notag
\end{align}
in $C^\infty$-topology, where $g^i(t)=g^i_t$.    Thus,   if  we let $v_i=(\varphi_{i})_*\tau_{j_0}\in \Gamma(M_i,$
$\newline K_{M_i}^{-ll_{j_0}})$ for some  large  integer   $l_{j_0}=\frac{1}{r_{j_0}^2}$ and $l=l(l_{j_0})\le N$,    then there exists a large integer  $i_0$  such that
for any $i\ge i_0$ it holds:
\begin{align}\label{l-norm-2}
&i')~ |v_i|_{h^{i}_{t_x}}\geq \frac{3}{8}, ~{\rm in}~ (\varphi_{i}\circ\psi_{j_0} )(V(x;2\delta) \bigcap B_x(3\delta))); \notag\\
&ii') \int_{M_{i}}|v_i|_{h^{i}_{t_x}}^2  {\rm dv}_{g_{t_x}^{i}}\leq \frac{5}{2}r_{j_0}^{2n-2}\frac{\eta}{l^{n-1}};\notag\\
&iii')~|\bar{\partial}_{J_{i}} v_i|_{h^{i}_{t_x}} \le\frac{1}{4}\eta ,  ~{\rm in}~  (\varphi_{i}\circ\psi_{j_0} )(V(x;\delta))  ;\notag\\
&iv')~  \int_{M_{i}}|\bar{\partial}_{J_{i}} v_i|^2_{h^{i}_{t_x}} {\rm dv}_{g_{t_x}^{i}}  \leq \frac{5}{4}r_{j_0}^{2n-2}\frac{\eta}{l^{n-1}}.
\end{align}
Here  $t_x=t_0' r_{j_0}^2/l$ and $h^{i_j}_{t_x}$  are  hermitian metrics of $K_{M_{i}}^{-ll_{j_0}}$  induced by
$g^{i}_{t_x}$.

 By solving   $\bar{\partial}$-equations for  $K_{M_{i}}^{-ll_{j_0}}$-valued (0,1)-form $\sigma_i$,
\begin{align}
\bar{\partial}\sigma_i=\bar{\partial}v_i,   ~{\rm in}~ M_{i}, \notag
\end{align}
we  get  the $L^2$-estimates   from   (\ref{L^2}) and $ iv')$  in (\ref{l-norm-2}),
\begin{align}\label{norm}
\|\sigma_i\|^2_{L^2(M_{i},g^{i}_{t_x})}\leq \frac{4}{ll_{j_0}}\int_{M_{i}}|\bar{\partial}_{J_{i}}v_i|^2 {\rm dv}_{g_{t_x}^{i}} \leq \frac{5\eta}{l^n  l_{j_0}^n}.
\end{align}
Hence,   by  (\ref{elliptic}) and  $ iii')$  in (\ref{l-norm-2}),  we derive
\begin{align}\label{small-c0-section}
&|\sigma_i |_{h^{i}_{t_x}}(q) \notag\\
&\leq 2C_2(\sup_{(\varphi_{i}\circ \psi_{j_0})(V(x;\delta))}|\bar{\partial}v_i|_{h^{i}_{t_x}}\notag\\
&+\delta^{-n} [(ll_{j_0})^{n}\int_{(\varphi_{i}\circ \psi_{j_0})(V(x;\delta))}
|\sigma_i|^2_{h^{i}_{t_x}}{\rm dv}_{g^{i}_{t_x}}]^\frac{1}{2})\notag\\
&\leq  2C_2 (\frac{1}{4}\eta+ \delta^{-n}[(ll_{j_0})^{n}\int_{M_{i}}|\sigma_i|^2{\rm dv}_{g^{i}_{t_x}}]^\frac{1}{2})\notag\\
&\leq 2C_2 ( \frac{1}{4}\eta+ \delta^{-n}[(ll_{j_0})^{n} \frac{5\eta}{l^n l_{j_0}}r_{j_0}^{2n-2}]^\frac{1}{2})\notag \\
&\le  5C_2( \frac{1}{4}\eta+ \delta^{-n}\sqrt{\eta})\leq \frac{1}{8}, ~ \forall~ q\in (\varphi_{i}\circ \psi_{j_0})(V(x;2\delta)).
\end{align}

Let $s_i=v_i-\sigma_i$.   Then  $s_i$  is a holomorphic section of $K_{M_{i}}^{-ll_{j_0}}$.   By $ i')$  in  (\ref{l-norm-2}) and (\ref{small-c0-section}),  we have
$$|s_i|_{h^{i}_{t_x}}(q_1)\geq \frac{3}{8}-\frac{1}{8}=\frac{1}{4} , ~\forall ~ q_1\in (\varphi_{i}\circ \psi_{j_0})(V(x;2\delta)\bigcap B_x(3\delta)).$$
Moreover,
 by   $ ii')$  in (\ref{l-norm-2}),  it is easy to see that
\begin{align}\label{l2-section-order}
\int_{M_{i}}|s_i|_{h_{t_x}^{i}}^2{\rm dv}_{g^{i}_{t_x}}&\leq 2(\int_{M_i}|v_i|_{h_{t_x}^{i}}^2{\rm dv}_{g^{i}_{t_x}}+\int_{M_{i}}
 |\sigma_i|_{h_{t_x}^{i}}^2{\rm dv}_{g^{i}_{t_x}})\notag\\
&\leq 4(2\pi)^n\frac{r_{j_0}^{2n}}{l^n}.
\end{align}
Thus by  the estimate (\ref{uniform-gradient}), we get
\begin{align}
\|\nabla s_i\|_{h^{i}_{t_x}}\leq \sqrt{4(2\pi)^n}C_1 \sqrt{l}r_{j_0}^{-1}.\notag
\end{align}
Since $d(p_{i},q_1)\leq 4\frac{r_{j_0}}{\sqrt{l}}\delta$,
we  deduce
\begin{align}
&|s_i(p_{i})|_{h^{i}_{t_x}}\geq |s_i(q_1)|-4\frac{r_{j_0}}{\sqrt{l}}\delta \|\nabla s_i\|_{h^{i}_{t_x}} \notag\\
&\geq |s_i|_{h^{i}_{t_x}}(q_1)-8\sqrt{(2\pi)^n}C_1\delta\geq \frac{1}{8}.\notag
\end{align}
This  proves  the theorem  while $l_x$ is chosen by $ll_{j_0}$.

\end{proof}

\vskip3mm

\section{Proof of Theorem \ref{main-theorem-wang-zhu}--I}

  In this section,  we  use the estimate in Section 5 to give a lower bound of  $\rho_{ll_0}(x)$ for a  sequence of almost K\"{a}hler-Einstein  manifolds.

\begin{theo}\label{almost-ke-1}Let $(M_i, g^i)$ be  a sequence of almost K\"ahler-Einstein manifolds as in Section 3 and $ (M_\infty, g_\infty)$ be its
  Gromov-Hasusdorff limit.  Then there exists  an integer $l_0>0$, which  depends  only  on   $ (M_\infty, g_\infty)$    such that for any integer  $l>0$ there exists  a uniform constant $c_l>0$ with property:
\begin{align}\label{partial estimate}
\rho_{ll_0}(M_{i},g^{i})\geq c_l.
\end{align}
\end{theo}

The proof of Theorem \ref{almost-ke-1} depends on the following lemma.

\begin{lem}\label{comparison}
Let  $(M, g)$ be a Fano manifold with $\omega_g\in 2\pi c_1(M)$ which  satisfies
\begin{align}\label{condition-fano}
{\rm Ric}(g)\geq -\Lambda^2 g~{\rm and}~  {\rm diam}(M,g)\leq D.
\end{align}
Let $g_t$  be    a solution of (\ref{Ricci flow}) with the initial metric $g$.   Then there exists  a small $t_0=t_0(l, \Lambda, D)$ such that the following is true:
if  $s\in \Gamma(M, K_M^{-l})$ is  a holomorphic section with
$\int_M|s|^2_{h_t}{\rm dv}_{g_t}=1$ for some $t\le t_0$ which satisfies
$$|s|^2_{h_t} (p)\geq c>0,$$
then
\begin{align}\label{begerman-3}
|s|^2_{h} (p)\geq c'>0~{\rm and}~
\int_M|s|^2_{h}{\rm dv}_{g}\leq c''.
\end{align}
Here    ${h_t}$ and $h$ are    hermitian metrics  of $K_{M_i}^{-l}$  induced by
$g_t$ and $g$, respectively,   and $c',c''>0$  are  uniform constants  depending only on $c,l, \Lambda$ and  $ D$.
\end{lem}

\begin{proof}
Let $\omega_{g_t}=\omega_g +\sqrt{-1}\partial\bar{\partial}\phi$. Namely,  $\phi$ are potentials of $g_t$. Then $\phi=\phi(x,t)$ satisfies
\begin{align}\label{potential-equation}
\frac{\partial}{\partial t}\phi=\log\frac{(\omega_g+\sqrt{-1}\partial\bar{\partial}\phi)^n}{\omega_g^n}+\phi-f_g,
\end{align}
where $f_g$ is  the  Ricci potential of $g$ normalized by
$$\int_M f_g{ \rm dv}_g^n=0.$$
  Since
  $$\Delta f_g=\rm{R}(g)-n\ge -(n-1)\Lambda^2-n, $$
  by using  the Green formula,   we see
  $$f_g(x)\le -\int_M G(x,\cdot) \Delta f_g\le C(\Lambda, D).$$
    Thus applying the maximum  principle to (\ref{potential-equation}),  it follows
 $$\phi\ge  - C(\Lambda, D).$$
 On the other hand, integrating  both sides of  (\ref{potential-equation}),  we have
\begin{align}
\frac{d}{dt}\int_M\phi {\rm dv}_{g}&=\int_M\log\frac{(\omega_g+\sqrt{-1}\partial\bar{\partial}\phi)^n}{\omega_g^n}{\rm dv}_{g}+\int_M\phi {\rm dv}_{g} -\int_M f_g {\rm dv}_{g}&\notag \\
&\leq \int_M\phi  {\rm dv}_{g}+ C,\notag
\end{align}
It follows
$$ \int_M\phi{\rm dv}_{g}\le Ce^{t}\le eC.$$
 Hence by using  the Green formula to $\phi$, we can also get
 $$\phi\le  C'(\lambda, D).$$
  As a consequence, we derive
  \begin{align}\label{hermitian-equivalent}  e^{-C'l}  |\cdot|_h\le |\cdot|_{h_t}=e^{-l\phi}|\cdot|_ h\le  e^{Cl}|\cdot|_h.
  \end{align}
    Therefore to prove Proposition \ref{comparison}, we suffice to prove

    \begin{claim}\label{claim-6}
Let $s\in \Gamma(M, K_M^{-l})$ be a holomorphic section.  Suppose that
\begin{align}
\int_M|s|^2_{h}{\rm{dv}}_g=1.\notag
\end{align}
Then
\begin{align}\label{begerman-5}
\int_M|s|^2_{h_t}{\rm{dv}}_{g_t}\geq c(l,\Lambda, D)>0.
\end{align}
\end{claim}

Since
\begin{align}\frac{\partial}{\partial t} \frac{(\omega_g+\sqrt{-1}\partial\bar{\partial}\phi)^n}{\omega_g^n}&=\Delta'\frac{\partial \phi}{\partial t}\notag\\
&=-R(g_t)+n\le \lambda=\lambda(\Lambda),\notag
\end{align}
 \begin{align}\label{upper-bound-volume}
\text{vol }_{g_t}(\Omega)\leq e^{\lambda t} \text{vol }_g(\Omega), ~\forall~\Omega\subset M.
\end{align}
It follows
\begin{align}\label{lower-bound-volume}\text{vol }_{g_t}(\Omega)&=V-\text{vol }_{g_t}(M\setminus \Omega)
\ge V- e^{\lambda t} \text{vol }_g(M\setminus\Omega)\notag\\
&\ge  \text{vol }_g(\Omega)-2V\lambda t.\end{align}
By  the estimate (\ref{l^2-estimate-section}), we see
$$|s(x)|_{h}^2\le H=H(\Lambda,D).$$
 Then
\begin{align}
 \int_0^H {\rm vol}_g \{x\in M|~|s(x)|_{h}^2\geq s\}ds=\int_M |s|^2_{h}{\rm dv}_g. \notag
\end{align}
Hence,  by using (\ref{hermitian-equivalent}),  and (\ref{upper-bound-volume}) and  (\ref{lower-bound-volume}), we get
\begin{align}
\nonumber \int_M|s|^2_{h_t}{ \rm dv}_{g_t}\ge & \int_0^H  {\rm vol}_{g_t} \{x\in M|~|s(x)|_{h_t}^2\geq s\}ds\\
\nonumber &\geq \int_0^H  {\rm vol}_{g_t} \{x\in M|~ |s(x)|_{h}^2\geq e^{C'l} s\}ds &\\
\nonumber &\geq e^{-C'l}\int_0^{e^{C'l}H}  [{\rm vol}_{g} \{x\in M|~|s(x)|_{h}^2\geq s\}-2V\lambda t]ds\\
\nonumber &\geq e^{-C'l} (1-2\lambda VHe^{C'l}t).
\end{align}
Therefore, by choosing  $t_0\le (4\lambda VHe^{C'l})^{-1}$,  we  derive  (\ref{begerman-5}). The claim is proved.

\end{proof}

\begin{proof}[Proof of the Theorem \ref {almost-ke-1}]  By Proposition \ref{partial},  we see that for any
$x\in M_\infty$ and a sequence $\{p_i\subset M_i\}$ which converges  to $x$, there exist  two large number $l_x$ and $i_0$,  a small time  $ t_x$  such that there exists  a holomorphic section  $s_i\in \Gamma (K_{M_i}^{-l_x}, h^i_{t_x})$  for any $i\ge i_0$  with  $\int_{M_i}|s_i|_{h^i_{t_x}}^2\rm{dv}_{g^i}\le 1$
which satisfies
\begin{align}
 |s_i|_{h_{t_x}^i}(p_i) \geq  \frac{1}{8}, \notag
 \end{align}
where $h^i_{t_x}$ is the hermitian metric of $K_{M_i}^{-l_x}$  induced by
$g^i_{t_x}$.
 By    Lemma \ref{comparison},    it follows that there exists a constant  $ c(l_x,\Lambda, D)$  and
  a holomorphic section  $\hat s_i\in \Gamma (K_{M_i}^{-l_x}, h_i)$  for any $i\ge i_0$  with  $\int_{M_i}|\hat s_i|_{h_i}^2\rm{dv}_{g^i}=1$
which satisfies
\begin{align}
 |\hat s_i|_{h_i}(p_i)\geq  c_x=c(l_x,\Lambda, D), \notag\end{align}
where $h_i$ is the hermitian metric of $K_{M_i}^{-l_x}$  induced by
$g^i$.

 Let    $C=C(C_S,n)$ be the constant as in (\ref{section-estiamte-1}),   which   depending only on $\Lambda$ and $ D$.
 For each $x$,  we choose $r_x=\frac{c_x}{2} l_x^{-\frac{n+1}{2}}C.$
 Then by the estimate  (\ref{section-estiamte-1}), we get
 \begin{align}
 |\hat s_i|_{h_i}(q)\geq  \frac{c_x}{2},~\forall ~q\in B_{p^i}(r_x). \notag\end{align}
Take $N$  balls  $B_{x_\alpha}(\frac{r_{x_\alpha}}{2})$ to cover $M_\infty$.  Then   it is easy to see that  there exists $i_1\ge i_0$ such that  $\cup_\alpha B_{p_{\alpha}^i}(r_{x_\alpha})=M_i$  for any $i\ge i_1$, where $\{ p_{\alpha}^i\}$  is a set of $N$ points in $M_i$.  This shows that for any $q\in M_i$ $(i\ge i_1)$ there exist a ball
$B_{p_{\alpha}^i}(r_{x_\alpha})$ and  a holomorphic section  $s^i_{\alpha}\in \Gamma (K_{M_i}^{-l_{x_\alpha}}, h_i)$ such that
 $q\in B_{p_{\alpha}}(r_{x_\alpha}^i)$, and
 $\int_{M_i}|s^i_{\alpha}|_{h_i}^2\rm{dv}_{g^i}=1$ and
 \begin{align}\label{begerman-4}
  | s_\alpha^i|_{h_i}(q)\geq  c=\min_{\alpha} \{c_{x_\alpha}\}>0.\end{align}
 Set $l_0=\prod_\alpha l_{x_\alpha}$.  Then  by using  a standard method (cf. \cite{DS},  \cite{T5}),
     for any  $q\in M_i$ $(i\ge i_1)$,  one can construct  another   holomorphic section
       $s\in \Gamma (K_{M_i}^{-l_0}, h_i)$  based on  holomorphic sections $s_\alpha^i$
   such that  $\int_{M^i}|s|_{h_i}^2\rm{dv}_{g^i}=1$ and
 \begin{align}
  | s|_{h_i}(q)\geq  c'>0,\notag
  \end{align}
where $c'=c'(l_0, c)$.
This  proves  the theorem  for $l=1$.   One can also prove the theorem for general multiple  $l\ge 1$ as above.
\end{proof}

\section{Proof of Theorem \ref{main-theorem-wang-zhu}--II }

In this section,  we prove Theorem \ref{main-theorem-wang-zhu} in case of almost K\"ahler-Ricci solitons.
We assume that   a Fano manifold $(M,g)$ admits  a non-trivial holomorphic vector field $X$,   where  $X$ lies in
an reductive  Lie subalgebra $\eta_r$ of space of holomorphic vector fields,  and   $g$ is $K_X$-invariant  with $\omega_g\in 2\pi c_1(M)$ \cite{TZ}.
We also suppose  that  $g$ satisfies  the following  geometric conditions:
\begin{align}\label{condition-ks}
&i) ~{\rm Ric} (g)+L_X g\geq -\Lambda^2g, ~ |X|_g\leq A ~\text{and}~\text{diam }(M,g)\leq D;\notag \\
&ii)~R(g)\geq -C_0.
\end{align}
In particular, under the condition i),  $g$ has a uniform $L^2$-Sobolev constant $C_s=C_s(\Lambda, A, D)$ (cf. \cite{WZ2}).
We note that the volume of $(M,g)$ is uniformly bounded below by the normalized condition $\omega_g\in 2\pi c_1(M)$ and
it is uniformly bounded above by the volume comparison theorem \cite{WW}.

 Now we consider  the  following modified K\"ahler-Ricci flow with  the above initial  K\"ahler metric $g$,
\begin{align}\left\{\begin{aligned}\label{modified-Ricci-flow}
&\frac{\partial}{\partial t}g=-{\rm{ Ric }}(g)+g+L_Xg,\\
&g_0=g(\cdot,0)=g.
 \end{aligned} \right.
\end{align}
Clearly,  solutions $g_t$  $(t\in (0,\infty))$ of (\ref{modified-Ricci-flow}) are  all $K_X$-invariant.

Since  the Sobolev constant $g$  is  uniformly bounded below,  by Zhang's result \cite{Zh},    we have an analogy to Lemma \ref{Sobolev} as follows.

   \begin{lem}\label{modified-Sobolev}
All solutions $g_t$  of (\ref{modified-Ricci-flow}) have  Sobolev constants $C_s=C_s(\Lambda, A,$
\newline $ D)$ uniformly bounded below.
Namely,  the following inequalities hold,
\begin{align}
(\int_M f^{\frac{2n}{n-1}}d{\rm{v}}_{g_t})^\frac{n-1}{n}\leq C_s(\int_M f^2(R+\hat C_0)d{\rm{v}}_{g_t}+\int_M|\nabla f|^2d{\rm{v}}_{g_t}),
\notag
\end{align}
where $f\in C^1(M)$ and $\hat C_0 $ is a uniform constant depending only on  the lower bound  $C_0$ of scalar curvature $R$ of $g$. \end{lem}

\begin{lem}\label{modified-iteration} Let $\Delta=\Delta_t$ be the Lapalace operator associated to $g_t$.  Suppose that $f\ge 0$ satisfies
\begin{align}\label{subsolution-2}(\frac{\partial}{\partial t}-(\Delta+X))f\leq a f, \end{align}
where  $a$ is a  constant.
Then  for any  $t\in (0,1) $, we have
 \begin{align}\label{c^0-estimate-soliton}
&\sup_{x\in M}f(x,t)\notag\\
&\leq \frac{C_1(\Lambda, A, D,C)}{t^\frac{n+1}{p}}(\int_{\frac{t}{2}}^t\int_M|f(x,\tau)|^pd\text{{\rm{v}}}_{g_\tau}d\tau)^\frac{1}{p}.
\end{align}
\end{lem}

\begin{proof}As in the proof of Lemma \ref{iteration} , multiplying both sides of (\ref{subsolution-2}) by $f^p$,  we have
\begin{align}
 &\int_M f^pf'_\tau {\rm dv}_{g_\tau}+p\int_M |\partial f|^2f^{p-1}{\rm dv}_{g_\tau}-\int_M \langle \partial \theta,\partial f\rangle f^p{\rm dv}_{g_\tau}\notag\\
&\leq a \int_M f^{p+1}{\rm dv}_{g_\tau}.\notag\end{align}
On the other hand, by (\ref{modified-Ricci-flow}), it is easy to see
\begin{align}
\int_M f^pf'_\tau {\rm dv}_{g_\tau} =\frac{1}{p+1}\frac{d}{d\tau}(\int_M f^{p+1} {\rm dv}_{g_\tau})+\frac{1}{p+1}\int_M (R-n-\Delta \theta)f^{p+1}
 d\text{v}_{g_\tau}.\notag
 \end{align}
 Thus we get
\begin{align}&\frac{1}{p+1}\frac{d}{d\tau}(\int_M f^{p+1} {\rm dv}_{g_\tau})+\frac{1}{p+1}\int_M (R-n)f^{p+1} d\text{v}_{g_\tau}+p\int_M |\partial f|^2f^{p-1}{\rm dv}_{g_\tau}\notag\\
 &\leq a\int_M f^p {\rm dv}_{g_\tau}.\notag
\end{align}
It follows
 \begin{align}\label{gradient-integral-2}
\frac{d}{d\tau}\int_M f^{p+1}   {\rm dv}_{g_{\tau}}+\int_M (R+\hat C_0)f^{p+1}   {\rm dv}_{g_{\tau}}+2\int_M |\nabla f^{\frac{p+1}{2}}|^2\notag\\
\leq ((p+1)a+n+C_0)\int_Mf^{p+1}  {\rm dv}_{g_{\tau}}.
\end{align}
Note that (\ref{gradient-integral-2}) is similar to (\ref{gradient-integral}).   Therefore,  we can follow the argument in the proof of Lemma \ref{iteration}
 to obtain (\ref{c^0-estimate-soliton}).

\end{proof}

Recall  that  according to \cite{WZ2} a sequence of weak almost  K\"ahler-Ricci solitons $(M_i, J_i, g^i, X_i)$ $(i\to\infty)$  satisfy  the condition i) in (\ref{condition-ks}) and
\begin{align}\label{almost-Ricci-soliton-integral}
iii)~ \int_{M_i} |{\rm Ric}(g^i)-g^i-L_{X_i}g^i|  {\rm dv}_{g^i}^n\rightarrow 0, ~\text{as}~i\to\infty.
\end{align}
As in \cite{WZ2},  we shall  further
assume  that  the solutions  $g^i_t$ of (\ref{condition-ks}) with the initial metrics $g^i$  satisfy
\begin{align}\label{further-condition-2}
&iii)~ |X^i|_{g^i_t}\leq \frac{B}{\sqrt{t}};\notag\\
&vi)~  \int_0^1dt\int_{M_i} |R(g^i_t)-\Delta\theta_{g^i_t}-n|  {\rm dv} _{g^i_t}^n\rightarrow 0, ~\text{as}~i\to\infty,
\end{align}
where $B$ is a uniform constant.
It was proved that under the conditions $i)$ of   (\ref{condition-ks}), and  (\ref{almost-Ricci-soliton-integral})  and (\ref{further-condition-2}) there exists a
subsequence of $\{(M_i, J_i, g^i, X_i)\}$ which converges  to a  K\"ahler-Ricci soliton away from  singularities of Gromov-Hausdorff limit  with codimension  4.

\begin{defi}\label{almost-kr-solitons}  $\{(M_i, J_i, g^i, X_i)\}$  are called   a sequence of  almost  K\"ahler-Ricci solitons if
 (\ref{condition-ks}), (\ref{almost-Ricci-soliton-integral})   and (\ref{further-condition-2}) are  satisfied.
\end{defi}

\begin{lem}\label{uniform-potentia-krl}
Let  $\{(M_i, J_i, g^i, X_i)\}$   be  a sequence of almost K\"{a}hler-Ricci solitons.
 Then there exists a uniform constant $C=C(\Lambda, D, B,C_0)$ such that for any $t\in (0, 1)$  there exists
  $N=N(t)$  such that for any $i\ge N$ it holds
$$|\nabla h^i_t|\leq  C~\text{and}~|R^i_t|\le C.$$
\end{lem}

\begin{proof} By
\begin{align}\label{h-evol}
&(\frac{\partial}{\partial t}-(\Delta+X))|\nabla(h-\theta)|^2\notag\\
&=-|\nabla \bar{\nabla}(h-\theta)|^2-|\nabla\nabla(h-\theta)|^2+|\nabla(h-\theta)|^2\notag\\
&\le |\nabla(h-\theta)|^2,
\end{align}
we apply  Lemma  \ref{modified-iteration} to get
\begin{align}&|\nabla(h-\theta)|^2\notag\\
&\le \frac{C}{t^{n+1}}\int_{\frac{t}{2}}^{t}\int_M|\nabla(h-\theta)|^2{\rm dv}_{g_\tau}d\tau\notag \\
&= \frac{C}{t^{n+1}}\int_{\frac{t}{2}}^{t}\int_M(\theta-h)\Delta(h-\theta){\rm dv}_{g_\tau}d\tau\notag \\
&\leq  \frac{C}{t^{n+1}}\int_{\frac{t}{2}}^{t}\int_M osc_M(h-\theta)|R-n-\Delta\theta|{\rm dv}_{g_\tau}d\tau.\notag
\end{align}
By (\ref{c0-u}),   it follows
\begin{align}\label{right2}
|\nabla(h-\theta)|^2\leq  \frac{C}{t^{(n+1)(n+\frac{3}{2})}}\int_{\frac{t}{2}}^{t}\int_M |R-n-\Delta\theta|{\rm dv}_{g_\tau}d\tau.
\end{align}

On the other hand,  by  the evolution equation of $(\Delta+X)(h-\theta)$ \cite{CTZ},
\begin{align}
&(\frac{\partial}{\partial t}-(\Delta+X))[(\Delta+X)(h-\theta)]\notag\\
&=(\Delta+X)(h-\theta)+|\nabla \bar{\nabla}(h-\theta)|^2, \notag
\end{align}
we have
\begin{align}
&(\frac{\partial}{\partial t}-(\Delta+X))[(\Delta+X)(h-\theta)+|\nabla(h-\theta)|^2]\notag\\
&\leq (\Delta+X)(h-\theta)+|\nabla(h-\theta)|^2\notag.
\end{align}
Then  applying Lemma  \ref{modified-iteration},  we get
\begin{align}\label{2-plus-estimate}
&(\Delta+X)(h-\theta)+|\nabla(h-\theta)|^2\notag\\
&\le\frac{C}{t^{n+1}}\int_{\frac{t}{2}}^{t}\int_M |(\Delta+X)(h-\theta)+|\nabla(h-\theta)|^2|{\rm dv}_{g_\tau}d\tau.
\end{align}
Note that by  iii) in (\ref{further-condition-2}) we have
\begin{align}
&\int_{\frac{t}{2}}^{t} \int_M|X(h-\theta)|{\rm dv}_{g_\tau}d\tau\notag \\
&\leq  B{\rm vol}(M) [\int_{\frac{t}{2}}^{t}\int_M|\nabla(h-\theta)|^2{\rm dv}_{g_\tau} d\tau]^{\frac{1}{2}}.\notag
\end{align}
It follows from  (\ref{c0-u}),
\begin{align}&\int_{\frac{t}{2}}^{t} \int_M |(\Delta+X)(h-\theta)+|\nabla(h-\theta)|^2|{\rm dv}_{g_\tau}\notag\\
&\le \int_{\frac{t}{2}}^{t} \int_M |R-n-\Delta\theta|{\rm dv}_{g_\tau} \notag\\
&+    C B{ \rm vol}(M)   \frac{1}{t^{\frac{1}{2}(n+1)(n+\frac{1}{2})}}
  [\int_{\frac{t}{2}}^{t} \int_M |R-n-\Delta\theta|{\rm dv}_{g_\tau} ]^{\frac{1}{2}}\notag\\
&+  C\frac{1}{t^{(n+1)(n+\frac{1}{2})}}  \int_{\frac{t}{2}}^{t} \int_M |R-n-\Delta\theta|{\rm dv}_{g_\tau}.\notag
\end{align}
Thus inserting the above inequality into (\ref{2-plus-estimate}), we derive
\begin{align}\label{2-plus-estimate-2}
&(\Delta+X)(h-\theta)+|\nabla(h-\theta)|^2\notag\\
&\leq  \frac{C}{t^{(n+1)(n+\frac{3}{2})}} (\int_{\frac{t}{2}}^{t}\int_M |R-n-\Delta\theta|{\rm dv}_{g_\tau}d\tau\notag\\
&+
 [\int_{\frac{t}{2}}^{t} \int_M |R-n-\Delta\theta|{\rm dv}_{g_\tau}d\tau ]^{\frac{1}{2}}).
 \end{align}

Combining  (\ref{right2}) and (\ref{2-plus-estimate-2}),  we see that
 for any $t\in (0, 1)$  there exists
  $N=N(t)$  such that
\begin{align}\label{R1}
|\frac{1}{\sqrt{t}}\nabla(h-\theta)|\le 1~\text{and}~R-n-\Delta\theta\le 1,~\forall~ i\ge N(t).
 \end{align}
It follows
  \begin{align}
 \Delta\theta =-|\nabla\theta|^2-X(h-\theta)-\theta\le C.\notag
\end{align}
As a consequence, we get $R\le C$, and so $|R|\le C$.

By (\ref {R1}), we have
 $$
\Delta\theta\geq R-n-1\geq -C.
$$
Thus
\begin{align}\label{gradient-theta}
 |\nabla\theta|^2=-X(h-\theta)-\theta-\Delta\theta\le C.
\end{align}
 Again by  (\ref {R1}),  we prove that
$|\nabla h|\leq C.$

\end{proof}

By  Lemma \ref{modified-Sobolev}  and  the scalar curvature estimate in  Lemma \ref{uniform-potentia-krl}, we see that for any $t\in (0, 1)$
 there exists  an  integer $N=N(t)$   such that the Sobolev constant $C_s$  of $g_t^i$ is uniformly bounded for any $i\ge N$.
 Then by the gradient  estimate of K\"ahler potentials   in  Lemma \ref{uniform-potentia-krl}, we can follow the arguments  in  Lemma   \ref{section-gradient-estimate}  and Lemma
\ref{gradient-section}  (also see Remark \ref{remark-1} and Remark \ref{remark-lemma-3.2}) to get an analogy of Proposition  \ref{corollary-1}.

\begin{prop}\label{uniform-soliton}
Let  $(M_i,g^i)$ be  a sequence of Fano manifolds with  almost K\"{a}hler-Ricci solitons  which satisfy (\ref{condition-ks}), (\ref{almost-Ricci-soliton-integral})  and (\ref{further-condition-2}).
 Then  for any $t\in (0, 1)$  there exist integers
  $N=N(t)$  such that for any $i\ge N$ and $l\ge l_0$   it holds,
\begin{align}\label{uniform-gradient-kr}
\|s\|_{h_t^i}+l^{-\frac{1}{2}}\|\nabla s\|_{h^i_t}\leq C l^{\frac{n}{2}}(\int_{M_i}|s|^2  {\rm dv}_{g^i_t})^{\frac{1}{2}}
\end{align}
and
\begin{align}\label{L^2-kr}
\int_{M^i}|v|_{h_t^i}^2\leq 4l^{-1}\int_{M_i}|\bar{\partial}\sigma|_{h_t^i}^2.
\end{align}
Here   $s\in H^0(M_i, K_{M_i}^{-l})$,   the norms of $|\cdot|_{h^i_t}$  are  induced by $g_t^i$,   and  the integer $l_0$ and  the uniform constant $C$  are independent  of $t$.

\end{prop}

By Proposition \ref{uniform-soliton},  we can  follow the arguments  in Proposition \ref{partial} and Theorem \ref{almost-ke-1} to prove

\begin{theo}\label{almost-ks}
Let $(M_i, g^i)$ be  a sequence of Fano manifolds with  almost K\"ahler-Ricci solitons and $ (M_\infty, g_\infty)$ be their  Gromov-Hasusdorff limit.  Then there exists  an integer $l_0>0$   which depending  only  on  $ (M_\infty, g_\infty)$  such that for any integer  $l>0$ there exists  a uniform constant $c_l>0$ with property:
\begin{align}\label{partial estimate}
\rho_{ll_0}(M_{i},g^{i})\geq c_l.
\end{align}
\end{theo}

\begin{proof}
We give a  sketch of  proof of Theorem  \ref{almost-ks}.

Step 1.  By  the rescaling method as in proof of  Proposition \ref{partial} with the helps of   Proposition \ref{uniform-soliton} and
 the pseudo-locallity theorem in \cite{WZ2},
   we  have an analogy of  Proposition \ref{partial}:   For any
sequence  of  $p_i\in M_i$ which converge to $x\in M_\infty$,  there  exist  two large number $l_x$ and $i_0$,  and
  a small time  $ t_x$    such that for any $i\ge i_0$  there exists a holomorphic section  $s_i\in \Gamma (K_{M_i}^{-l_x}, h_{t_x}^i)$ which satisfies
 \begin{align}\label{l2-norm-less-1}
 \int_{M_i}|s_i|_{h^i_{t_x}}^2{\rm dv}_{g^i_{t_x}}\le 1~{\rm and}~
 |s_i|_{h^i_{t_x}}(p_i) \geq  \frac{1}{8},
\end{align}
 where $g^i_t$ is  a   solution of (\ref{modified-Ricci-flow})  with the initial metric $g^i$ and  $h^i_{t_x}$ is the hermitian metric of $K_{M_i}^{-l_x}$  induced by
$g^i_{t_x}$.

Step 2. We can compare the $C^0$-norm of holomorphic  sections   with respect to  the varying  metrics  $g_t$ evolved in  the flow (\ref{modified-Ricci-flow}).
    In fact, we have

\begin{lem}\label{comparison1}
Let  $(M, g)$ be a Fano manifold with $\omega_g\in 2\pi c_1(M)$ which  satisfies (\ref{condition-ks}), and $g_t$     a solution of (\ref{modified-Ricci-flow}) with the initial metric $g$.   Then there exists  a small $t_0=t_0(l, \Lambda, D)$ such that the following is true:
if  $s\in \Gamma(M, K_M^{-l})$ is  a holomorphic section with
\begin{align}\label{norm-condition}
\int_M|s|^2_{h_t}{\rm{dv}}_{g_t}=1
\end{align} for some $t\le t_0$ which satisfies
\begin{align}\label{peak-condition}
|s|_{h_t} (p)\geq c>0,
\end{align}
then there is a holomorphic section $s'$ of $K_M^{-l}$ which satisfies
\begin{align}
|s'|_{h} (p)\geq c'>0 ~{\rm and}~\int_M|s'|^2_{h}{\rm dv}_{g}\le c'',\notag
\end{align}
where    ${h_t}$ and $h$ are   the hermitian metrics  of $K_{M}^{-l}$  induced by
$g_t$ and $g$, respectively,   and  the constants  $c'$  and $c''$  depend only on $c,l, \Lambda$, $A$, $C_0$ and  $ D$.
\end{lem}

\begin{proof}[Proof of Lemma \ref{comparison1}] Let   $\Phi_t$  be a one-parameter subgroup generated by $-X$.  Then $\Phi_t^*g_t$ is  a  solution of (\ref{Ricci flow}).   It is clear that
 (\ref{norm-condition}) also holds for $\Phi_t^*s, \Phi_t^*g_t, \Phi_t^*h_t$ and  the condition (\ref{peak-condition}) is equivalent to
$|\Phi_t^*s|_{\Phi_t^*h_t}(\Phi_{-t}(p))\geq c.$
Since the  Green functions  associated to  the metric $g$  is    bounded below  under the condition  $i)$ of (\ref{condition-ks}) (cf. \cite{Ma}, \cite{CTZ}),   we can
 follow the argument  in  Lemma  \ref{comparison} for   the metrics $\Phi_t^*g_t$ to obtain
 \begin{align}
|\Phi_t^*s|_{h} (\Phi_{-t}(p))\geq \tilde c~{\rm and}~
\int_M|\Phi_t^*s|^2_{h}{\rm dv}_{g}\leq  c'',\notag
\end{align}
where  the  constant   $ \tilde c$   depends only on $c,l, \Lambda$, $A$  and  $ D$.
  Let  $s'=\Phi_t^*s$.
Then by the gradient estimate of $|\nabla s'|\leq C(l,\Lambda,D,C_0,A)$, we have
\begin{align}
|s'|_{h} (p)\geq |s'|_{h} (\Phi_{-t}(p))-C(\Lambda,D,C_0,A)At\geq c'.\notag
\end{align}
This proves Lemma \ref{comparison1}.
\end{proof}

Step 3.  By using the covering argument as in Theorem \ref{almost-ke-1} together  with the results in Step 1 and Step 2, we can finish the proof of Theorem \ref{almost-ks}.

\end{proof}

\vskip3mm

\section{Proof of Corollary \ref{algebaric-structure}}

In this section, for simplicity, we just  give a proof of Corollary \ref{algebaric-structure} in case of almost  K\"{a}hler-Einstein manifolds with
 dimension  $n\ge 2$.
We assume that  a sequence of almost K\"{a}hler-Einstein manifolds $(M_i,g^i)$ with   a limit $(M_\infty,g_\infty)$ in Goromov-Hausdorff topology satisfies  the partial $C^0$-estimate,
 \begin{align}\label{partial-estimate}
\rho_l(M_i,g^i)\geq c_l>0,
\end{align}
 for some integer  $l$.
Then,  as an application of (\ref{partial-estimate}),  we have
\begin{align}\label{cohomology}
 H^0(M_i,K_{M_i}^{-m})\subseteq H^0(M_i,K_{M_i}^{-(m-l)})\otimes H^0(M,K_{M_i}^{-l}),
 \end{align}
 where  $m\geq l(n+2+[\Lambda^2])$  is any  integer and the constant $-\Lambda^2$ is a uniform lower bound of Ricci curvature of $(M_i,g^i)$ (cf. Proposition 7,  \cite{L1})\footnote{There is a generalization of  (\ref{cohomology}) under the Bakry-Em\'ery Ricci curvature  condition in Appendix.}

  We need a strong version of (\ref{partial-estimate}) as follows.

\begin{lem}\label{strong-version}
 For two different  points  $x, y\in M_\infty$, there exist $\ell=\ell(n,\Lambda, D,$
 \newline $x,y)$,  which is a multiple of $l$,  and  two
 sections  $s_x,s_y\in H^0(M_i,K_{M_i}^{-\ell})$ such that
\begin{align}\label{separate}
|s_x(p_i)|_{h_i }=|s_y(q_i)|_{h_i}=1~{\rm and}~ s_x(q_i)= s_y(p_i)=0,
\end{align}
where $p_i\rightarrow x, q_i\rightarrow y$.
 \end{lem}

\begin{proof}
 As in the proof of Proposition \ref{partial}, we can choose two compact sets $V(x;\delta_1^x), V(y;\delta_1^y)$ in $C_x$ and $C_y$,  respectively, such that
 $\phi_i\circ \psi_j(V(x;\delta_1^x))$ and  $\phi_i\circ\psi_j(V(y;\delta_1^y))$ are disjoint as long as  $j$ and $i$ are large enough.
 Let   $v_i^x,\sigma_i^x,s^i_x\in \Gamma(M_i, K_{M_i}^{-l_x})$ and $v_i^y,\sigma_i^y,s^i_y\in \Gamma(M_i, K_{M_i}^{-l_y})$  be sections associated  $x$ and $y$,  respectively. We may assume that $l_x=l_y=\ell$ for a multiple of $l$.  Moreover, by
 the $C^0$-estimate of $\sigma_i^x$ in $V(x;\delta^x)$  in (\ref{small-c0-section}), we see that $|s^i_x(q_i)|$ is small. Similarly, $|s^i_y(p_i)|$ is also small. Now we  define holomorphic sections
 \begin{align}
\tilde{s}^i_x=s^i_x-\frac{s^i_x(q_i)}{s^i_y(q_i)}s^i_y~{\rm and}~ \tilde{s}^i_y=s^i_y-\frac{s^i_y(p_i)}{s^i_x(p_i)}s^i_x.
 \end{align}
 Clearly, $\tilde s_x(q_i)= \tilde s_y(p_i)=0.$
 Then  $s_x=\frac{\tilde{s}^i_x}{|\tilde{s}^i_x(p_i)|_{h_i}}$
 and  $s_y=\frac{\tilde{s}^i_y}{|\tilde{s}^i_y(q_i)|_{h_i}}$  will satisfy (\ref{separate}).

 \end{proof}

By Lemma \ref{strong-version}, we prove

\begin{prop}\label{variety}
 Let $\{(M_i,g^i)\}$ be a sequence of Fano manifolds with Ricci bounded from below and diameter bounded from above,
 and $(M_\infty,g_\infty)$ its limit in Goromov-Hausdorff topology. Suppose that
(\ref{partial-estimate})  and  (\ref{separate})  in Lemma \ref{strong-version}  hold.  Then  $M_\infty$ is homeomorphic to an algebraic   variety.
\end{prop}

\begin{proof} By (\ref{partial-estimate}),
for any $k$, we can define holomorphisms
\begin{align}
T_{kl,i}:M_i\rightarrow \mathbb CP^N,\notag
\end{align}
where $N+1={\rm dim}  H^0(M_i,K_{M_i}^{-kl})$  is constant if  $i$ is large enough.
 Since $T_{kl,i}$ is uniformly Lipschitz by (\ref{section-estiamte-1}),  we get a limit map
\begin{align}
T_{kl,\infty}: M_\infty\rightarrow \mathbb  CP^N.\notag
\end{align}
 On the other hand, the images $W^{kl}_i$  of $T_{kl,i}$
  have  a  chow limit $W^{kl}$, which  coincides with  the image of  the map $T_{kl,\infty}$. Thus  $T_{kl,\infty}$ maps $M_\infty$ onto $W^{kl}=T_{kl,\infty}
   (M_\infty)$.
We claim that $T_{(n+2+[\Lambda^2])l,\infty}$ is injective, so the proposition is proved.

 By Lemma \ref{separate},   for any $x,y\in M_\infty$,  there are $p_i\rightarrow x$ and $q_i\rightarrow y$, and $s_{x},s_{y}\in H^0(M_\infty, K_{M_i}^{-k_1l})$ for some $k_1$ such that
\begin{align}\label{disjoint-points}
|s_{x}|_{h_i}(p_i)=|s_{y}|_{h_i}(q_i)|=1~{\rm and}~ s_{x}(q_i)=s_{y}(p_i)=0.
\end{align}
This  means $T_{k_1l,\infty}(x)\neq T_{k_1l,\infty}(y)$. We further show that
 \begin{align}\label{assumption}
T_{(n+2+[\Lambda^2])l,\infty}(x)\neq T_{(n+2+[\Lambda^2])l,\infty}(y).
\end{align}
In fact, if (\ref {assumption}) is not true,  it is easy to see
 $T_{il,\infty}(x)=T_{il,\infty}(y)$ for any $i\leq n+2+[\Lambda^2]$.
Then by (\ref{cohomology}), it follows
$$T_{kl,\infty}(x)= T_{kl,\infty}(y),~\forall~ k, $$
  which is  contradict to (\ref{disjoint-points}).
Thus  (\ref{assumption}) is true.  Hence  $T_{(n+2+[\Lambda^2])l,\infty}$ must be injective.

\end{proof}

\begin{proof}[Proof of  Corollary \ref{algebaric-structure}]  By the  Gromov compactness theorem,  there exists a subsequence  $\{(M_{i_k},g^{i_k})\}$
of $\{(M_{i},g^{i})\}$,  which converges to $(M_\infty,g_\infty)$.   Then  i) and ii) in Corollary \ref{algebaric-structure}  follow  from  a generalized   Cheeger-Colding-Tian compactness theorem  for a sequence of  almost K\"{a}hler-Einstein  manifolds \cite{TW} ( or   a  sequence of Fano manifolds with  almost K\"{a}hler-Ricci solitons  \cite{WZ2}).  Thus we suffice  to prove the part iii).
By  Proposition \ref{variety}, we know that $M_\infty$ is homomorphic to an algebraic variety  $W^{k_0l}$,   where $k_0= n+2+[\Lambda^2]$.  We
  further  show  that  $W^{k_0l}$ is a   log terminal $Q$-Fano variety.

Let  $H^0(M_\infty, K_{M_\infty}^{-k_0l})$ be a space of  bounded holomorphic  sections of $K_{\mathcal{R}}^{-k_0l}$ with respect to the induced metric  $g_\infty$. Then for any compact set $K\subseteq \mathcal{R}\subseteq M_\infty$, we know that there are $t_K>0$  and $K_i\subseteq M_i$
  such that $(K_i,g_i(t_K))$ converge to $(K, g_\infty)$ smoothly. Thus by the argument in Proposition \ref{partial} and Lemma \ref{comparison},
we can identify $H^0(M_\infty, K_{M_\infty}^{-k_0l})$ with the limit of $H^0(M_i,K_{M_i}^{-k_0l})$.  But, from the proof in Proposition \ref{variety},
the latter is the same as $H^0(W^{k_0l},$
\newline $\mathcal{O}_{\mathbb CP^N}(1))$.
 This  implies that $M_\infty$  is homeomorphic to the  normalization of $W^{k_0l}$ since the  codimension of singularities of  $W^{k_0l}$ is at least  $2$ (cf. \cite{T5} and \cite{DS}). Hence  $W^{kl_0}$ is normal. By  \cite{BBEGZ}, it remains to prove that $W^{k_0l}$  is a  $Q$-Fano variety.

Let $\mathcal{S}= {\rm Sing}(M_\infty)$, $\mathcal{\hat S}=T_{k_0l,\infty}(\mathcal{ S} )$,  and let $ W_s\subset\mathcal{\hat S}$ be the singular set of   $W^{k_0l}$. Then both $ W_s$ and  $\mathcal{\hat S}$ lie in a subvariety of  $W^{k_0l}$  with codimension at least 2.   Thus we suffice  to prove that  $W_s=\mathcal{\hat S}$ since $(W^{k_0l}, \mathcal{O}_{\mathbb CP^N}(1)) =$
$\newline K^{- k_0l}_{W^{k_0l}\setminus\mathcal{\hat S}}$.   In the following, we give
 a  proof  for the general limit K\"ahler-Ricci soliton  $(M_\infty, g_\infty)$  in Section 7 by using PDE method as in \cite{DS}.  Namely,  $g_\infty$ satisfies
an equation,
\begin{align}\label{kr-soliton-limit} {\rm Ric}(g_\infty)-g_\infty-L_{X_\infty}g_\infty=0,~{\rm in }~M_\infty\setminus\mathcal{S},
\end{align}
where $X_\infty$ is the limit holomorphic vector field  of $(M_i,X_i)$ on $M_\infty\setminus\mathcal{S}$   \cite{WZ2}.

 On contrary, we suppose that  $W_s\neq\mathcal{\hat S}$.  Then  there exists some $x\in\mathcal{S}$ such that  $p=T_{k_0l,\infty}(x)\in W^{k_0l}\setminus W_s $,  a  smooth point in  $W^{k_0l}$.  Thus  there exists a small ball $B$ around $p$ in  $W^{k_0l}$ with the standard holomorphic coordinates such that  the induced K\"ahler form  $\omega_0=\frac{1}{k_0l}\omega_{g_{FS}}$ by  the Fubini-Study metric $g_{KS}$ of the projective space  is smooth on $B$.
We may assume that $\omega_{0}=\sqrt{-1}\partial\bar{\partial} v$ for some K\"ahler potential $v$ on $B$.

Let $\rho_\infty$ be the limit of $\rho_{k_0l}(M_i,g^i)$ (perhaps replaced by a subsequence of $\rho_{k_0l}(M_i,g^i)$) on
$(M_\infty\setminus\mathcal{ S}  , g_\infty)$. Then  $\rho_\infty$ and $|\nabla \rho_\infty|_{g_\infty}$ are both  uniformly bounded since  $\rho_{k_0l}(M_i,g^i)$ and $|\nabla \rho_{k_0l}(M_i,g^i)|_{g^i}$ are all    uniformly bounded by (\ref{section-estiamte-1}).   Clearly, $\rho_\infty$ satisfies
 $$\omega_{g_\infty}=\omega_{0}+\sqrt{-1}\partial\bar{\partial}\rho_\infty, ~{\rm in }~W^{k_0l}\setminus \mathcal{\hat S}.$$
Let $u=v+\rho_\infty$. Then by (\ref{kr-soliton-limit}),  we see that $u$ satisfies
 \begin{align}
\sqrt{-1}\partial\bar{\partial}(\log {\rm det} (u_{i\bar{j}})+X_\infty(u)+u)=0, ~{\rm in }~ B\setminus \mathcal{\hat S}.\notag
\end{align}
It follows
\begin{align}\label{local-equation}
\log {\rm det} (u_{i\bar{j}})+X_\infty(u)+u=const., ~{\rm in }~ B\setminus \mathcal{\hat S}.
\end{align}

 We claim that there  exists a  uniform $C$ such that
\begin{align}\label{equivalence-metric}C^{-1}\delta_{i\overline j}\le u_{i\overline j}\le C\delta_{i\overline j}, ~{\rm in } B\setminus \mathcal{\hat S}.
\end{align}

 Since  the  basis  in  $H^0(M_\infty, K_{M_\infty}^{-k_0l})$,  which gives the embedding $T_{k_0l,\infty}$,  is uniformly $C^1$-bounded, we have
$$\omega_{0}\le C\omega_{g_\infty}, ~{\rm in }~M_\infty.$$
On the other hand, by (\ref{gradient-theta}),
$$|X_\infty(\rho_\infty)|\le|X_\infty|_{g_\infty}|\nabla \rho_\infty|_{g_\infty}\le C,~{\rm in }~M_\infty. $$
Then $X_\infty(u)$ is uniformly bounded.
Thus by (\ref{local-equation}), we see that
$\log {\rm det} (u_{i\bar{j}})$ is uniformly positive and bounded. This implies (\ref{equivalence-metric}).

By the above claim,  we can apply the following lemma to show that   $u$ is a smooth function in a small neighborhood of  $p$.  But
this is impossible by $x\in \mathcal{ S}$. Hence  $W^{k_0l}$ must be  a $Q$-Fano variety.

\end{proof}

\begin{lem}  Let $u$ be a smooth solution of (\ref{local-equation}) in $B\setminus \mathcal{\hat S}$, where $B$ is a ball in the euclidean space in $\mathbb C^n$ and $\mathcal{\hat S}$ is a closed subset in $\mathbb C^n$  with real dimension less than $2n-1$. Suppose that $u$ satisfies (\ref{equivalence-metric}).
Then $u$  can be extended to  a smooth function on $\frac{1}{4}B$.
\end{lem}

\begin{proof}
 By the Schaulder estimate for the  equation (\ref{local-equation}), we suffices  to get  a  $C^{2,\alpha}$-regularity of $u$ in $\frac{1}{4}B$.  We  first  do  the $C^{1,1}$-estimate.

 For any $0<\epsilon< \frac{1}{8}$ and any unit vector  $v$, we let the difference quotient
 $$w=w_\epsilon=\frac{u(x+\epsilon v)+u(x-\epsilon v)-2u(x)}{\epsilon^2}.$$
  Then  by the convexity of $\log {\rm det}$,  we  get from (\ref{local-equation}),
\begin{align}\label{interpo-1}
u^{i\bar{j}}w_{i\bar{j}}\geq e^g \frac{g(x+\epsilon v)+g(x-\epsilon v)-2g(x)}{\epsilon^2},
\end{align}
where $g=-u-X_\infty(u)$.  Denote $(a_{\alpha\beta})$ to be the  $2n\times 2n$ matrix of Riemannian metric of $g_\infty$
and $(a^{\alpha\beta})={\rm det}(a_{\delta\gamma})(a_{\alpha\beta})^{-1}.$  It is clear that (\ref{interpo-1}) is equivalent to
\begin{align}\label{interpo-2}(a_{\alpha\beta}w_\beta)_\alpha \geq l(x)+\frac{h(x+\epsilon v)-h(x)}{\epsilon}, ~{\rm in }~\frac{3}{4}B\setminus\mathcal{\hat S},
\end{align}
where $l=f(x+\epsilon v)\frac{e^g (x+\epsilon v)-e^g(x)}{\epsilon}, h=e^gf$ and $f=\frac{g(x)-g(x-\epsilon v)}{\epsilon}$.
 Note that $X_\infty$ can be extended to a holomorphic vector field on $B$. Then by (\ref{equivalence-metric}), $w$ can be regarded as a weak sub-solution in  (\ref{interpo-2}) in whole $\frac{3}{4}B$.  Thus by  the $L^\infty$-estimate  arising from the  Moser iteration,  we have,
\begin{align}\label{harnack}
\sup_{\frac{1}{2}B}(w_\epsilon ) \leq C(|w_\epsilon|_{L^p(\frac{3}{4}B)}+|l|_{L^\frac{q}{2}(B)}+|h|_{L^q(B)}),
\end{align}
where $C$ depends  only on  $(a_{\alpha\beta})$,  $p\ge 1$ and $q>2n$.
In fact, by  Theorem  8.17 in \cite{GT}, the estimate (\ref{harnack}) holds for sub-solution $w$ as follows,
$$(a_{\alpha\beta}w_\beta)_\alpha\ge l+ <v, Dh>.$$
  But   Theorem  8.17 is also true when  the term $<v, Dh>$ is  replaced by the difference quotient $\frac{h(x+\epsilon v)-h(x)}{\epsilon}$.

Since $g$ is uniformly  Lipschitz in $B\setminus\mathcal{\hat S}$, $l,h$ are $L^\infty$-functions in $B$. On the other hand, by (\ref{equivalence-metric}),  $u\in W^{2,p}(\frac{3}{4}B)$ for any $p\ge 1$, and  so $|w_\epsilon|_{L^p(\frac{3}{4}B)}$ is uniformly bounded.  Thus the (\ref{harnack}) implies that $w_\epsilon$ is  uniformly  bounded  above.  As a consequence,  $C^{1,1}$-derivative $u_{vv}$  is uniformly  bounded above. By (\ref{equivalence-metric}),  we can also get a uniform  lower  bound  of  $u_{vv}$. Hence $C^{1,1}$-norm of $u$ is uniformly  bounded in $\frac{1}{2}B$.

  Next to get $C^{2,\alpha}$-estimate of $u$ in (\ref{local-equation}), we can apply  Evans-Krylov theorem,  Theorem 17.14 in \cite{GT} to $C^{1,1}$-solution of (\ref{local-equation}) in $\frac{1}{2}B$ directly.  This is because   (\ref{local-equation}) is strictly elliptic in $B$ and $-u-X_\infty(u)$ is Lipschitz. Thus the lemma is proved.

\end{proof}

\vskip3mm

\section{Conclusion}

In the proofs  of   Theorem \ref{almost-ke-1} and  Theorem \ref{almost-ks},  the constants $c_l$ in the estimates (\ref{almost-ke-1})  and (\ref{almost-ks}) may depend on
the limit $(M_\infty, g_\infty)$.  In this section,  we   show that $c_l$ just depends on $n, l_0$ and $l$, and  the  geometric  uniform constants  $\Lambda$ and $D$  in  $i)$ of (\ref{almost-ke-condition}), or  the  constants $ \Lambda,  D, C_0$ and $ B$ in (\ref{condition-ks}) and $iii)$ of  (\ref{further-condition-2}).    Thus we  complete  the proof of Theorem \ref{main-theorem-wang-zhu}.  For simplicity, we just consider the case of  almost K\"{a}hler-Einstein  Fano manifolds below.

Set   a class  of  Fano manifolds by
$$\mathcal{K}_{\Lambda,D}=\{(M^n,g)|~\omega_g\in 2\pi c_1(M),  {\rm Ric }(g)\geq - (n-1)\Lambda^2, {\rm diam }(M, g)\leq D\}.$$
It is known that $\mathcal{K}_{\Lambda,D}$ is precompact in Gromov-Hausdorff topology.  Moreover, by Cheeger-Colding theory  in \cite{CC}, any Gromov-Hausdorff limit
$M_\infty$  in  $\mathcal{K}_{\Lambda,D}$  contains   singularities  with   codimension  at least 2 and each  tangent cone  at  $x\in M_\infty$  is a metric cone $C_x$, which  also contains   singularities  with   codimension  at least 2.

Let $\mathcal{K}_{\Lambda,D}^0$ be a subset of $\mathcal{K}_{\Lambda,D}$ such that ${\mathcal H}^{2n-2}( {\rm Sing}(C_x))=0$ for any $x\in  M_{\infty}$,  where $M_\infty$ is any  Gromov-Hausdorff limit  in $\mathcal{K}_{\Lambda,D}^0$.   Then according  to the proofs in  Proposition \ref{partial} and Theorem \ref{almost-ke-1}, we have

\begin{prop} Let $(M,g)\in\mathcal{K}_{\Lambda,D}^0$ and $g_t$ a solution of (\ref{Ricci flow}) with the initial metric $g$.
Then there exist   a small number $\delta=\delta(\Lambda, D,n)$ and  a large integer  $ l_0=l_0(n, \Lambda, D)$ such that  the following is true:
  if  $g$  satisfies
\begin{align}\label{integer-condition-l1}\int_0^1\int_M|R-n|{\rm dv}_{g_t} dt\leq \delta,
\end{align}
then for any  integer $l$  there exists  a uniform constant  $c=c(n, l,\Lambda, D)>0$  such that
\begin{align}\label{pe}
\rho_{ll_0}(M,g)\geq c.
\end{align}

\end{prop}

\begin{proof}
 By  Theorem \ref{almost-ke-1}, we  see  that for any $Y \in \bar{\mathcal{K}}_{\Lambda,D}^0$,  there exist a small number  $\delta_Y>0$, a large integer $ l_Y$ and
 a uniform  constant $c_Y>0$ such that if $M\in \mathcal{K}_{\Lambda,D}$ satisfies
\begin{align}
 {\rm d}_{GH}((M,g),(Y, g_Y))\leq \delta_Y,\int_0^1\int_M|R-n|{\rm dv}_{g_t} dt\leq \delta_Y,\notag
\end{align}
then $$\rho_{l_Y}(M,g)\geq c_Y.$$
 Since $\bar{\mathcal{K}}_{\Lambda,D}$ is compact, we can cover it by finite  balls $B_{Y_i}(\delta_{Y_i})(1\leq i\leq N)$  in Gromov-Hausdroff topology.
 Putting $l_0=\Pi l_{Y_i}, \delta=\min\{\delta_{Y_i}\}$ and $c=\min\{c_{Y_i}\}$. Then we get (\ref{pe})  for $l=1$,   if $(M,g)$ satisfies (\ref{integer-condition-l1}).
 (\ref{pe}) is also true for general $l$ as in the proof of  Theorem \ref{almost-ke-1}.
\end{proof}

 (\ref{lower-bound-bergman-tian}) in
Theorem  \ref{main-theorem-wang-zhu} follows from (\ref{pe}).

\vskip3mm
\section{Appendix }

In this appendix,  we use the following  Siu's lemma  to generalize the finite generation formula  (\ref{cohomology}) under the Bakry-Em\'ery Ricci curvature condition \cite{Si}.

\begin{lem}
Let $(M^n,g)$ be a compact complex manifold, $G$ a holomorphic line bundle, $E$ a holomorphic line bundle with a hermitian metric $e^{-\psi}$
 whose Ricci curvature is positive. Let  $\{s_i \}_{1\leq i\leq p}$ be a basis of  $ H^0(M,G)$ and $|s|^2=\Sigma_{i=1}^p |s_i|^2$.   Then for any  $f\in H^0(M,(n+k+1)G+E+K_M)$ which  satisfies
$$ \int_M\frac{|f|^2e^{-\psi}}{|s|^{2(n+k+1)}} d{\rm v}_g< +\infty,$$
 there are some $h_i\in H^0(M,(n+k)G+E+K_M)$ ($k\geq 1$)  such that $f=\Sigma_{i=1}^p  h_j\otimes s_j$ and  each $h_i$ satisfies
$$  \int_M\frac{|h_j|^2e^{-\psi}}{|s|^{2(n+k)}} d{\rm v}_g\leq \frac{n+k}{k} \int_M\frac{|f|^2e^{-\psi}}{|s|^{2(n+k+1)}}d{\rm v}_g.$$
\end{lem}

\begin{prop}\label{finite}
Let $(M,g)$ be a K\"{a}hler manifold
with $${\rm Ric}(g)+{\rm Hess}\, u \geq -Cg,$$ where
$X=\nabla u$ is a holomorphic vector field and $|u|\leq A$. Assume  that
\begin{align}\label{partial assumption}
c'\geq \rho_l(M,g)\geq c>0
\end{align} for some $l\in \mathbb{N}$.  Then for any $s\in H^0(M,K_M^{-m})$ with $m\geq (n+2)l+C+1$, there are $u_i\in H^0(M,K_M^{-(m-l)})$
such that $s=\Sigma_{i=0}^Nu_i\otimes s_i$, where $\{s_i\}$  is an orthonormal basis of $H^0(M,K_M^{-l})$.  Moreover,  each $u_i$ satisfies
\begin{align}\label{effective}
\int_M |u_i|^2_{h^{\otimes m-l}}{\rm dv }_g\leq (n+1)e^{2A}(\frac{c'}{c})^{\frac{m}{l}}\int_M |s|^2_{h^{\otimes m}}{\rm dv }_g.
\end{align}

\end{prop}
\begin{proof}
Putting $L=K_{M}^{-1}$ and $m-C-1=(n+k+1)l+r$ $ (0\leq r <l)$, we decompose $mL$ as
$$mL= (n+k+1)(lL)+((m-(n+k+1)l)L-K_M)+K_M.$$
Let  $h$ and $\omega_g^n$   be  two hermitian metrics on $L$ such that
$${\rm Ric}(L, h) =g, \, {\rm Ric}(L,\omega_g^n)={\rm Ric}(g).$$
Denote  the line bundle $(m-(n+k+1)l)L-K_M$ by $E$. Then    $h_1=h^{\otimes m-(n+k+1)l}\otimes e^{-u}\otimes \omega_g^n$ is a  hermitian  metric on $E$.
It is easy to see
  $${\rm Ric}(E,h_1)=(m-(n+k+1)l)\omega_g+{\rm Ric}(g)+\sqrt{-1}\partial\bar{\partial} u\geq \omega_g.$$
Now applying the above lemma to $G=lL$, $s_i$, $E$ and $f=s$, we  see that there are  $u_i\in H^0(M,(n+k)G+E+K_M)$ such that
 $$\int_M\frac{|u_i|^2_{h^{\otimes (n+k)l}\otimes h_1}}{(\Sigma_{i=0}^N |s_i|^2_{h^{\otimes l}})^{n+k}} {\rm dv }_g\\
\leq \frac{n+k}{k}\int_M \frac{|s|^2_{h^{\otimes (n+k+1)l}\otimes h_1}}{(\Sigma_{i=0}^N |s_i|^2_{h^{\otimes l}})^{n+k+1}}{\rm dv }_g.$$
The above  is equivalent to
\begin{align}
\int_M\frac{|u_i|^2_{h^{\otimes m-l}}}{(\Sigma_{i=0}^N |s_i|^2_{h^{\otimes l}})^{n+k}}e^{-u}{\rm dv }_g\leq \frac{n+k}{k}\int_M\frac{|s|^2_{h^{\otimes m}}}{(\Sigma_{i=0}^N |s_i|_{h^{\otimes l}}^2)^{n+k+1}}e^{-u}{\rm dv }_g.\notag
\end{align}
By (\ref{partial assumption}),  it follows
\begin{align}
\frac{1}{e^{2A} c'^{n+k}}\int_M  |u_i|^2_{h^{\otimes m-l}}{\rm dv }_g\leq \frac{n+k}{kc^{n+k+1}}\int_M |s|^2_{h^{\otimes m}}{\rm dv }_g,\notag
\end{align}
which implies (\ref{effective}) immediately.
\end{proof}

\end{document}